\documentclass{amsart}

\usepackage{amsmath}
\usepackage{amsthm}
\usepackage{amssymb}
\usepackage{graphicx}

\theoremstyle{plain}
\newtheorem{theorem}{Theorem}[section]
\newtheorem{lemma}{Lemma}[section]
\newtheorem{proposition}{Proposition}[section]
\newtheorem{corollary}{Corollary}[section]

\theoremstyle{remark}
\newtheorem{remark}{Remark}[section]
\newtheorem{example}{Example}[section]

\numberwithin{equation}{section}

\newcommand{\C}{\mathbb{C}}
\newcommand{\D}{\mathbb{D}}
\newcommand{\N}{\mathbb{N}}

\newcommand{\T}{\mathbb{T}}
\newcommand{\Z}{\mathbb{Z}}

\newcommand{\p}{\mathbf{p}}

\newcommand{\mcB}{\mathcal{B}}
\newcommand{\mcF}{\mathcal{F}}

\newcommand{\barp}{\overline{p}}

\DeclareMathOperator{\cspn}{\overline{\mathrm{sp}}}

\title[\null]{Explicit formulas for the inverses of Toeplitz matrices, with applications}

\author[\null]{AKIHIKO INOUE}

\address{A.\ Inoue\\
Department of Mathematics\\
Hiroshima University\\
Higashi-Hiroshima 739-8526\\
Japan}
\email{inoue100@hiroshima-u.ac.jp}

\begin{document}

\subjclass[2010]{Primary 60G10; secondary 65F05, 60G25, 15B05.}

\keywords{Toeplitz matrix, Toeplitz system, explicit formula, closed-form formula, 
linear-time algorithm.}


\begin{abstract}
We derive novel explicit formulas for the inverses of truncated block Toeplitz matrices 
that correspond to a multivariate minimal stationary process. 
The main ingredients of the formulas are the Fourier coefficients of the phase function 
attached to the spectral density of the process. 
The derivation of the formulas is based on a recently developed 
finite prediction theory applied to 
the dual process of the stationary process. 
We illustrate the usefulness of the formulas by two applications. 
The first one is a strong convergence result for solutions of 
general block Toeplitz systems for a multivariate short-memory process. 
The second application is closed-form formulas for 
the inverses of truncated block Toeplitz 
matrices corresponding to a multivariate ARMA process. 
The significance of the latter is that they provide us 
with a linear-time algorithm 
to compute the solutions of corresponding block Toeplitz systems.
\end{abstract}

\maketitle


\section{Introduction}\label{sec:1}

Let $\T:=\{z\in\C :\vert z\vert=1\}$ be the unit circle in $\C$. 
We write $\sigma$ for the normalized Lebesgue measure $d\theta/(2\pi)$ on
$([-\pi,\pi), \mcB([-\pi,\pi)))$, where $\mcB([-\pi,\pi))$ is the Borel $\sigma$-algebra
on $[-\pi,\pi)$; thus we have $\sigma([-\pi,\pi))=1$.
For $p\in [1,\infty)$, we write $L_p(\T)$ for the Lebesgue space of measurable functions $f:\T\to\C$
such that $\Vert f\Vert_p<\infty$, where 
$\Vert f\Vert_p:=\{\int_{-\pi}^{\pi}\vert f(e^{i\theta})\vert^p \sigma(d\theta)\}^{1/p}$. 
Let $L_p^{m\times n}(\T)$ be the space of $\C^{m\times n}$-valued functions on 
$\T$ whose entries belong to $L_p(\T)$.

Let $d\in\N$. For $n\in\N$, we consider the block Toeplitz matrix
\[
T_n(w)
:=\left(
\begin{matrix}
\gamma(0)   & \gamma(-1)  & \cdots & \gamma(-n+1)\cr
\gamma(1)   & \gamma(0)   & \cdots & \gamma(-n+2)\cr
\vdots      & \vdots      & \ddots & \vdots      \cr
\gamma(n-1) & \gamma(n-2) & \cdots & \gamma(0)
\end{matrix}
\right)
\in \C^{dn\times dn},
\]
where
\begin{equation}
\gamma(k):=\int_{-\pi}^{\pi}e^{-ik\theta}w(e^{i\theta})\frac{d\theta}{2\pi} \in \C^{d\times d}, 
\qquad k\in\Z,
\label{eq:gamma123}
\end{equation}
and the symbol $w$ satisfies the following two conditions:
\begin{gather}
\mbox{$w\in L^{d\times d}_1(\T)$ and $w(e^{i\theta})$ is a positive Hermitian matrix $\sigma$-a.e.,}
\label{eq:A}\\
w^{-1}\in L^{d\times d}_1(\T).
\label{eq:M}
\end{gather}
Let $\{X_k:k\in\Z\}$ be a $\C^d$-valued, centered, weakly stationary process 
that has spectral density $w$, hence autocovariance function $\gamma$. 
Then the conditions (\ref{eq:A}) and (\ref{eq:M}) imply that $\{X_k\}$ is {\it minimal}\/ 
(see Section 10 of \cite[Chapter II]{Roz}).

In this paper, 
we show novel explicit formulas for $T_n(w)^{-1}$ 
(Theorem \ref{thm:TforXwithM123}), which are especially useful for large $n$ (see \cite{BS99}). 
The formulas are new even for $d=1$. 
The main ingredients of the formulas are the 
Fourier coefficients of $h^*h_{\sharp}^{-1}=h^{-1}h_{\sharp}^*$, where 
$h$ and $h_{\sharp}$ are $\C^{d\times d}$-valued outer functions on
$\T$ such that
\begin{equation}
w(e^{i\theta}) 
= h(e^{i\theta}) h(e^{i\theta})^*
= h_{\sharp}(e^{i\theta})^*h_{\sharp}(e^{i\theta}), \qquad
\qquad \mbox{$\sigma$-a.e.}
\label{eq:decomp888}
\end{equation}
(see \cite{HL}; see also Section \ref{sec:2}). 
We note that the unitary matrix valued function $h^*h_{\sharp}^{-1}=h^{-1}h_{\sharp}^*$ on $\T$ 
attached to $w$ 
is called the {\it phase function\/} of $w$ 
(see page 428 in \cite{Pe}).

Let $\{X_k\}$ be as above, and let $\{X^{\prime}_k: k\in\Z\}$ be the 
dual process of $\{X_k\}$ (see \cite{M60}; see also Section \ref{sec:2} below). 
In the proof of the above explicit formulas for $T_n(w)^{-1}$, 
the dual process $\{X^{\prime}_k\}$ plays an important role. 
In fact, the key to the proof of the explicit formulas for $T_n(w)^{-1}$ 
is the following equality (Theorem \ref{thm:Tdual123}):
\begin{equation}
\left(T_n(w)^{-1}\right)^{s,t} = \langle X^{\prime}_s, P_{[1,n]}X^{\prime}_t\rangle,
\qquad s, t \in \{1,\dots,n\}.
\label{eq:Tdual123}
\end{equation}
Here, $\langle \cdot, \cdot \rangle$ stands for the Gram matrix 
(see Section \ref{sec:3}) and 
$P_{[1,n]}X^{\prime}_t$ denotes the best 
linear predictor of $X^{\prime}_t$ based on the observations $X_{1},\dots,X_{n}$ 
(see Section \ref{sec:2} for the precise definition). 
Moreover, for $n\in\N$, $A \in \C^{dn \times dn}$ and $s, t\in\{1,\dots,n\}$, 
we write $A^{s,t}\in \C^{d\times d}$ for the $(s,t)$ block of $A$; thus 
$A = (A^{s,t})_{1\le s, t\le n}$. 
The equality (\ref{eq:Tdual123}) enables us to apply the $P_{[1,n]}$-related methods 
developed in \cite{I00}, \cite{I08}, \cite{IK06}, \cite{IKP1}, \cite{IKP2} and others to derive 
the explicit formulas for $T_n(w)^{-1}$.

We illustrate the usefulness of the explicit formulas for $T_n(w)^{-1}$ by two applications. 
The first one is a strong convergence result for solutions of block Toeplitz systems. 
For this application, we assume (\ref{eq:A}) as well as the following condition:
\begin{equation}
\mbox{$\sum_{k=-\infty}^{\infty} \Vert \gamma(k)\Vert <\infty$ and 
$\min_{z\in\T}\det w(z)>0$.}
\label{eq:S}
\end{equation}
Here, for $a\in\C^{d\times d}$, $\Vert a\Vert$ denotes the 
operator norm of $a$. 
The condition (\ref{eq:S}) implies that $\{X_k\}$ with spectral density $w$ is a 
{\it short-memory}\/ process. 
We note that (\ref{eq:M}) follows from (\ref{eq:A}) and (\ref{eq:S}) 
(see Section \ref{sec:4}). 
Under (\ref{eq:A}) and (\ref{eq:S}), for $n\in\N$ and a $\C^{d\times d}$-valued sequence 
$\{y_k\}_{k=1}^{\infty}$ such that $\sum_{k=1}^{\infty} \Vert y_k\Vert < \infty$, 
let
\begin{equation}
Z_n=(z_{n,1}^{\top},\dots,z_{n,n}^{\top})^{\top}\in \C^{dn\times d}\ \ \mbox{with}\ \ 
z_{n,k}\in\C^{d\times d},\ k\in\{1,\dots,n\},
\label{eq:Zz-314}
\end{equation}
be the solution to the block Toeplitz system
\begin{equation}
T_n(w)Z_n = Y_n,
\label{eq:TS183}
\end{equation}
where
\begin{equation}
Y_n := (y_1^{\top},\dots,y_n^{\top})^{\top}\in\C^{dn\times d}.
\label{eq:Rr-314}
\end{equation}
Also, let
\begin{equation}
Z_{\infty}=(z_{1}^{\top},z_{2}^{\top},\dots)^{\top}\ \ \mbox{with}\ \ 
z_{k}\in\C^{d\times d},\ k \in \N,
\label{eq:Zzinfty-314}
\end{equation}
be the solution to the corresponding infinite block Toeplitz system
\begin{equation}
T_{\infty}(w)Z_{\infty} = Y_{\infty},
\label{eq:TSinfty183}
\end{equation}
where
\begin{equation}
T_{\infty}(w) := 
\left(
\begin{matrix}
\gamma(0)  &  \gamma(-1)  &  \gamma(-2)  &   \cdots   \cr
\gamma(1)  &  \gamma(0)   &  \gamma(-1)  &   \cdots   \cr
\gamma(2)  &  \gamma(1)   &  \gamma(0)   &   \cdots   \cr
\vdots     &  \vdots      &  \vdots      &   \ddots   
\end{matrix}
\right)
\label{eq:Tinfty123}
\end{equation}
and
\begin{equation}
Y_{\infty} := (y_1^{\top},y_2^{\top},\dots)^{\top}.
\label{eq:Rrinfty-314}
\end{equation}
Then, our result (Theorem \ref{thm:Bax-conv123}) reads as follows:
\begin{equation}
\lim_{n\to\infty} \sum_{k=1}^n \Vert z_{n,k} - z_k\Vert = 0.
\label{eq:conv-z-234}
\end{equation}

We explain the background of the result (\ref{eq:conv-z-234}). 
As above, let $\{X_k:k\in\Z\}$ be a $\C^d$-valued, centered, weakly stationary process 
that has spectral density $w$. 
For $n\in\N$, the finite and infinite predictor coefficients $\phi_{n,k}\in\C^{d\times d}$, 
$k\in\{1,\dots,n\}$, and $\phi_k$, $k\in\N$, of $\{X_k\}$ are defined by
\[
P_{[1,n]} X_{n+1} = \sum_{k=1}^n \phi_{n,k} X_{n+1-k}
\qquad \mbox{and} \qquad
P_{(-\infty,n]} X_{n+1} = \sum_{k=1}^{\infty} \phi_{k} X_{n+1-k},
\]
respectively; see Section \ref{sec:3} for the precise definitions of 
$P_{[1,n]}$ and $P_{(-\infty,n]}$. 
We note that $\sum_{k=1}^{\infty} \Vert \phi_k\Vert < \infty$ holds under (\ref{eq:A}) and (\ref{eq:S}) 
(see Section \ref{sec:4} below and (2.16) in \cite{IKP2}). 
{\it Baxter's inequality}\/ in \cite{Bax}, \cite{CP}, \cite{HD} states that, 
under (\ref{eq:A}) and (\ref{eq:S}), 
there exists $K\in (0,\infty)$ such that
\begin{equation}
\sum_{k=1}^{n}\Vert \phi_{n,k} - \phi_k \Vert
\le K\sum_{k=n+1}^{\infty}\Vert \phi_k\Vert,\qquad
n\in\N.
\label{eq:Baxter597}
\end{equation}
In particular, we have
\begin{equation}
\lim_{n\to\infty} \sum_{k=1}^n \Vert \phi_{n,k} - \phi_k\Vert = 0.
\label{eq:pred-234}
\end{equation}
If we put $\tilde{w}(e^{i\theta}):=w(e^{-i\theta})$, 
then, $(\phi_{n,1},\dots,\phi_{n,n})$ is the solution to the block Toeplitz system
\[
T_n(\tilde{w})(\phi_{n,1},\dots,\phi_{n,n})^* = (\gamma(1),\dots,\gamma(n))^*,
\]
called the {\it Yule--Walker equation}, while $(\phi_{1},\phi_{2}, \dots)$ 
is the solution to the corresponding 
infinite block Toeplitz system
\[
T_{\infty}(\tilde{w})(\phi_{1},\phi_{2},\dots)^* = (\gamma(1),\gamma(2),\dots)^*.
\]
Clearly, $\tilde{w}$ satisfies (\ref{eq:A}) and (\ref{eq:S}) since so does $w$. 
Therefore, our result (\ref{eq:conv-z-234}) can be viewed as an extension to (\ref{eq:pred-234}). 
It should be noted, however, that we prove (\ref{eq:conv-z-234}) directly, 
without proving an analogue of Baxter's inequality (\ref{eq:Baxter597}).

The convergence result (\ref{eq:pred-234}) has various applications in time series analysis, such as 
the autoregressive sieve bootstrap (see, e.g., \cite{IKP2} and the references therein), whille 
Toeplitz systems of the form (\ref{eq:TS183}) appear in various fields, such as filtering of signals. 
Therefore the extension (\ref{eq:conv-z-234}), as well as the other results explained below, 
may potentially be useful in such fields. 
We note that Baxter's inequality (\ref{eq:Baxter597}), hence (\ref{eq:pred-234}), is also proved for 
univariate and multivariate FARIMA (fractional autoregressive 
integrated moving-average) processes, which are 
long-memory processes, in \cite{IK06} and \cite{IKP2}, respectively. 
The FARIMA processes have singular spectral densities $w$ but 
our explicit formulas for $T_n(w)^{-1}$ above also cover them since we only assume 
minimality in the formulas. Applications of 
the explicit formulas to univariate and multivariate FARIMA processes will be discussed 
elsewhere. However, 
the problem of proving results of the type (\ref{eq:conv-z-234}) for 
FARIMA processes remains unsolved so far.

The second application of the explicit formulas for $T_n(w)^{-1}$ is 
closed-form formulas for $T_n(w)^{-1}$ with rational $w$ that 
corresponds to a univariate ($d=1$) or multivariate ($d\ge 2$) ARMA 
(autoregressive moving-average) process 
(Theorem \ref{thm:TforR123}). 
More precisely, we assume that 
$w$ is of the form
\begin{equation}
w(e^{i\theta})=h(e^{i\theta})h(e^{i\theta})^*,\qquad \theta\in [-\pi,\pi),
\label{eq:farima529}
\end{equation}
where $h:\T\to \C^{d\times d}$ satisfies the following condition:
\begin{equation}
\begin{aligned}
&\mbox{the entries of $h(z)$ are rational functions in $z$ that have}\\
&\mbox{no poles in $\overline{\D}$, and $\det h(z)$ has no zeros in $\overline{\D}$.}
\end{aligned}
\label{eq:C}
\end{equation}
Here $\overline{\D}:=\{z\in\C :\vert z\vert\le 1\}$ is the closed unit disk in $\C$. 
The closed-form formulas for $T_n(w)^{-1}$ consist of several building block matrices 
that are of fixed sizes independent of $n$. 
The significance of the formulas for $T_n(w)^{-1}$ is that they provide us with 
a linear-time, or $O(n)$, algorithm to compute the solution 
$Z\in \C^{dn\times d}$ to the block Toeplitz system
\begin{equation}
T_n(w)Z = Y
\label{eq:TS234}
\end{equation}
for $Y\in\C^{dn\times d}$ (see Section \ref{sec:6}). 
The famous Durbin--Levinson algorithm solves the equation (\ref{eq:TS234}) for 
more general $w$ in $O(n^2)$ time. 
Algorithms for Toeplitz linear systems that run faster than $O(n^2)$ 
are called {\it superfast}. While our algorithm is restricted to the class of $w$ 
corresponding to ARMA processes, the class is important in applications, 
and the linear-time algorithm is ideally superfast in the sense that there is no 
algorithm faster than $O(n)$.

Toeplitz matrices appear in a variety of fields, 
including operator theory, orthogonal polynomials on the unit circle, 
time series analysis, engineering, and physics. 
Therefore, there is a vast amount of literature on Toeplitz matrices. 
Here, we refer to \cite{BS99}, \cite{BS06}, 
\cite{GF}, \cite{GS}, \cite{Si05a}, \cite{Si05b} and \cite{Si11} 
for a textbook treatment. For example, in \cite[III]{GF}, 
the Gohberg-Semencul formulas in \cite{GSe}, which express 
the inverse of a Toeplitz matrix 
as a difference of products of lower and upper triangular Toeplitz 
matrices, are explained.

After this work was completed, the author learned of \cite{SRY} 
by Subba Rao and Yang, where they also provide an explicit series expansion for $T_n(w)^{-1}$ 
that corresponds to a univariate stationary process satisfying some conditions 
(see \cite{SRY}, Section 3.2). The main aim of \cite{SRY} is to reconcile the Gaussian and Whittle 
likelihood, and the series expansion in \cite{SRY} 
is tailored to this purpose, using the {\it complete DFT}\/ 
(discrete Fourier transform) introduced in \cite{SRY}. It should be noticed that $T_n(w)^{-1}$ appears in 
the Gaussian likelihood, while the Whittle likelihood is based on the ordinary DFT. 
Since most results of the present paper 
directly concern $T_n(w)^{-1}$, some of them may also be useful for studies related to the 
Gaussian likelihood.

This paper is organized as follows. 
We state the explicit formulas for $T_n(w)^{-1}$ 
in Section \ref{sec:2}. 
In Section \ref{sec:3}, we first prove (\ref{eq:Tdual123}) and then 
use it to prove the explicit formulas for $T_n(w)^{-1}$. 
In Section \ref{sec:4}, we prove (\ref{eq:conv-z-234}) for $w$ satisfying 
(\ref{eq:A}) and (\ref{eq:S}), using the explicit formulas for $T_n(w)^{-1}$. 
In Section \ref{sec:5}, we prove the closed-form formulas for $T_n(w)^{-1}$ with 
$w$ satisfying (\ref{eq:C}), using the explicit formulas for $T_n(w)^{-1}$. 
In Section \ref{sec:6}, we explain how the results in Section \ref{sec:5} 
give a linear-time algorithm to compute the solution to (\ref{eq:TS234}). 
Finally, the Appendix contains the omitted proofs of two lemmas.


\section{Explicit formulas}\label{sec:2}

Let $\C^{m\times n}$ be the set of all complex $m\times n$ matrices;
we write $\C^d$ for $\C^{d\times 1}$. 
Let $I_n$ be the $n\times n$ unit matrix.
For $a\in \C^{m\times n}$, $a^{\top}$ denotes the transpose of $a$, and
$\overline{a}$ and
$a^*$ the complex and Hermitian conjugates of $a$, respectively;
thus, in particular, $a^*:=\overline{a}^{\top}$. 
For $a\in\C^{d\times d}$, we write $\Vert a\Vert$ for the operator norm of $a$:
\[
\Vert a\Vert:=\sup_{u\in \C^d, \vert u\vert \le 1}\vert au\vert.
\]
Here 
$\vert u\vert:=(\sum_{i=1}^d\vert u^i\vert^2)^{1/2}$ denotes the Euclidean norm of
$u=(u^1,\dots,u^d)^{\top}\in \C^d$. 
For $p\in [1,\infty)$ and $K\subset \Z$, $\ell_p^{d\times d}(K)$ denotes the space of 
$\C^{d\times d}$-valued sequences 
$\{a_k\}_{k\in K}$ such that $\sum_{k\in K}\Vert a_k\Vert^p<\infty$. 
We write $\ell_{p+}^{d\times d}$ for $\ell_p^{d\times d}(\N\cup\{0\})$ 
and $\ell_{p+}$ for $\ell_{p+}^{1\times 1}=\ell_p^{1\times 1}(\N\cup\{0\})$.

Recall $\sigma$ from Section \ref{sec:1}. 
The Hardy class $H_2(\T)$ on $\T$ is the closed subspace of
$L_2(\T)$ consisting of $f\in L_2(\T)$ such that
$\int_{-\pi}^{\pi}e^{im\theta}f(e^{i\theta})\sigma(d\theta)=0$ for $m=1,2,\dots$.
Let $H_2^{m\times n}(\T)$ be the space of $\C^{m\times n}$-valued functions on
$\T$ whose entries belong to $H_2(\T)$.
Let $\D:=\{z\in\C : \vert z\vert<1\}$ be the open unit disk in $\C$.
We write $H_2(\D)$ for the Hardy class on $\D$, consisting of
holomorphic functions $f$ on $\D$ such that
$\sup_{r\in [0,1)}\int_{-\pi}^{\pi}\vert f(re^{i\theta})\vert^2\sigma(d\theta)<\infty$.
As usual, we identify each function $f$ in $H_2(\D)$ with its boundary function
$f(e^{i\theta}):=\lim_{r\uparrow 1}f(re^{i\theta})$, $\sigma$-a.e.,
in $H_2(\T)$.
A function $h$ in $H_2^{d\times d}(\T)$ is called \textit{outer}\/ if $\det h$ is a
$\C$-valued outer function, that is, $\det h$ satisfies
$\log\vert \det h(0)\vert
=\int_{-\pi}^{\pi}\log\vert \det h(e^{i\theta})\vert \sigma(d\theta)$
(see Definition 3.1 in \cite{KK}).

We assume that $w$ satisfies (\ref{eq:A}) and (\ref{eq:M}). 
Then $\log \det w$ is in $L_1(\T)$ (see Section 3 in \cite{IKP2}). Therefore
$w$ has the decompositions (\ref{eq:decomp888}) 
for two outer functions $h$ and $h_{\sharp}$ belonging to $H_2^{d\times d}(\T)$, and
$h$ and $h_{\sharp}$ are unique up to constant unitary factors 
(see Chapter II in \cite{Roz} and Theorem 11 in \cite{HL}; see also Section 3 in \cite{IKP2}). 
We may take $h_{\sharp}=h$ for the case $d=1$ but there is no such simple
relation between $h$ and $h_{\sharp}$ for $d\ge 2$. 
We define the outer function $\tilde{h}$ in $H_2^{d\times d}(\T)$ by
\begin{equation}
\tilde{h}(z) := \{h_{\sharp}(\overline{z})\}^*.
\label{eq:outft628}
\end{equation}
All of $h^{-1}$, $h_{\sharp}^{-1}$ and $\tilde{h}^{-1}$ also belong to $H_2^{d\times d}(\T)$ since 
we have assumed (\ref{eq:M}).

We define four $\C^{d\times d}$-valued sequences $\{c_k\}$, $\{a_k\}$, $\{\tilde{c}_k\}$ and $\{\tilde{a}_k\}$ by
\begin{align}
h(z)&=\sum_{k=0}^{\infty}z^kc_k,\qquad z\in\D,
\label{eq:MA111}\\
-h(z)^{-1}&=\sum_{k=0}^{\infty}z^ka_k,\qquad z\in\D,
\label{eq:AR111}\\
\tilde{h}(z)&=\sum_{k=0}^{\infty}z^k\tilde{c}_k,\qquad z\in\D,
\label{eq:MA222}
\end{align}
and
\begin{equation}
-\tilde{h}(z)^{-1}=\sum_{k=0}^{\infty}z^k\tilde{a}_k,\qquad z\in\D,
\label{eq:AR222}
\end{equation}
respectively. By (\ref{eq:M}), 
all of $\{c_k\}$, $\{a_k\}$, $\{\tilde{c}_k\}$ and $\{\tilde{a}_k\}$ belong to $\ell_{2+}^{d\times d}$.

We define a $\C^{d\times d}$-valued sequence $\{\beta_k\}_{k=-\infty}^{\infty}$ as the 
(minus of the) Fourier coefficients of the phase function $h^*h_{\sharp}^{-1}=h^{-1}h_{\sharp}^*$:
\begin{equation}
\begin{aligned}
\beta_k 
&=
-\int_{-\pi}^{\pi}e^{-ik\theta} h(e^{i\theta})^* h_{\sharp}(e^{i\theta})^{-1} \frac{d\theta}{2\pi}\\
&=
-\int_{-\pi}^{\pi}e^{-ik\theta} h(e^{i\theta})^{-1} h_{\sharp}(e^{i\theta})^* \frac{d\theta}{2\pi}, 
\qquad k \in \Z.
\end{aligned}
\label{eq:beta-def667}
\end{equation}
For $n\in\mathbb{N}$, $u \in \{1,\dots,n\}$ and $k\in\mathbb{N}$, we can define 
the sequences $\{b_{n,u,\ell}^k\}_{\ell=0}^{\infty}\in \ell_{2+}^{d\times d}$
by the recursion
\begin{equation}
\left\{
\begin{aligned}
b_{n,u,\ell}^1&=\beta_{u+\ell},\\
b_{n,u,\ell}^{2k}
&=\sum_{m=0}^{\infty} b_{n,u,m}^{2k-1} \beta_{n+1+m+\ell}^*,\qquad 
b_{n,u,\ell}^{2k+1}
=\sum_{m=0}^{\infty} b_{n,u,m}^{2k} \beta_{n+1+m+\ell}
\end{aligned}
\right.
\label{eq:recurs123}
\end{equation}
(see Section \ref{sec:3} below). 
Similarly, for $n\in\mathbb{N}$, $u \in \{1,\dots,n\}$ and $k\in\mathbb{N}$, we can define 
the sequences $\{\tilde{b}_{n,u,\ell}^k\}_{\ell=0}^{\infty}\in \ell_{2+}^{d\times d}$
by the recursion
\begin{equation}
\left\{
\begin{aligned}
\tilde{b}_{n,u,\ell}^1&=\beta_{n+1-u+\ell}^*,\\
\tilde{b}_{n,u,\ell}^{2k}
&=\sum_{m=0}^{\infty} \tilde{b}_{n,u,m}^{2k-1} \beta_{n+1+m+\ell},\qquad
\tilde{b}_{n,u,\ell}^{2k+1}
=\sum_{m=0}^{\infty} \tilde{b}_{n,u,m}^{2k} \beta_{n+1+m+\ell}^*.
\end{aligned}
\right.
\label{eq:til-recurs123}
\end{equation}

Recall from Section \ref{sec:1} that $(T_n(w)^{-1})^{s,t}$ denotes 
the $(s,t)$ block of $T_n(w)^{-1}$. 
Since $T_n(w)$, hence $T_n(w)^{-1}$, is self-adjoint, we have 
\begin{equation}
(T_n(w)^{-1})^{s,t} = ((T_n(w)^{-1})^{t,s})^*, \qquad s, t \in \{1,\dots,n\}.
\label{eq:SA-123}
\end{equation}
We use the following notation:
\[
s\vee t:=\max(s, t), \qquad s\wedge t:=\min(s, t).
\]

We are ready to state the explicit formulas for $(T_n(w))^{-1}$.

\begin{theorem}\label{thm:TforXwithM123}
We assume (\ref{eq:A}) and (\ref{eq:M}). Then the following two assertions hold.
\begin{enumerate}
\item For $n\in\N$ and $s, t\in\{1,\dots,n\}$, we have
\begin{equation}
\begin{aligned}
&\left(T_n(w)^{-1}\right)^{s,t} 
= \sum_{\ell = 1}^{s\wedge t} \tilde{a}_{s - \ell}^* \tilde{a}_{t - \ell} \\
&\qquad  + \sum_{u=1}^t \sum_{k=1}^{\infty} 
     \left\{ 
          \sum_{\ell = 0}^{\infty} \tilde{b}_{n,u,\ell}^{2k-1} a_{n + 1 - s + \ell} 
       +  \sum_{\ell = 0}^{\infty} \tilde{b}_{n,u,\ell}^{2k}   \tilde{a}_{s + \ell}
     \right\}^* \tilde{a}_{t-u}.
\end{aligned}
\label{eq:TforXwithM222}
\end{equation}
\item 
For $n\in\N$ and $s, t\in\{1,\dots,n\}$, we have
\begin{equation}
\begin{aligned}
&\left(T_n(w)^{-1}\right)^{s,t} 
= \sum_{\ell=s\vee t}^n a_{\ell - s}^* a_{\ell - t} \\
&\qquad  + \sum_{u=t}^n \sum_{k=1}^{\infty} 
     \left\{ 
          \sum_{\ell = 0}^{\infty} b_{n,u,\ell}^{2k-1} \tilde{a}_{s + \ell} 
       +  \sum_{\ell = 0}^{\infty} b_{n,u,\ell}^{2k}   a_{n + 1 - s + \ell}
     \right\}^* a_{u-t}.
\end{aligned}
\label{eq:TforXwithM111}
\end{equation}
\end{enumerate}
\end{theorem}

The proof of Theorem \ref{thm:TforXwithM123} will be given in Section \ref{sec:3}.

\begin{corollary}\label{cor:TforXwithM456}
We assume (\ref{eq:A}) and (\ref{eq:M}). Then the following two assertions hold.
\begin{enumerate}
\item For $n\in\N$ and $s, t\in\{1,\dots,n\}$, we have
\begin{equation}
\begin{aligned}
&\left(T_n(w)^{-1}\right)^{s,t} 
= \sum_{\ell = 1}^{s\wedge t} \tilde{a}_{s - \ell}^* \tilde{a}_{t - \ell} \\
&\qquad  + \sum_{u=1}^s \tilde{a}_{s-u}^* \sum_{k=1}^{\infty} 
     \left\{ 
          \sum_{\ell = 0}^{\infty} \tilde{b}_{n,u,\ell}^{2k-1} a_{n + 1 - t + \ell} 
       +  \sum_{\ell = 0}^{\infty} \tilde{b}_{n,u,\ell}^{2k}   \tilde{a}_{t + \ell}
     \right\}.
\end{aligned}
\label{eq:TforXwithM444}
\end{equation}

\item For $n\in\N$ and $s, t\in\{1,\dots,n\}$, we have
\begin{equation}
\begin{aligned}
&\left(T_n(w)^{-1}\right)^{s,t} 
= \sum_{\ell=s\vee t}^n a_{\ell - s}^* a_{\ell - t} \\
&\qquad  + \sum_{u=s}^n a_{u-s}^* \sum_{k=1}^{\infty} 
     \left\{ 
          \sum_{\ell = 0}^{\infty} b_{n,u,\ell}^{2k-1} \tilde{a}_{t + \ell} 
       +  \sum_{\ell = 0}^{\infty} b_{n,u,\ell}^{2k}   a_{n + 1 - t + \ell}
     \right\}.
\end{aligned}
\label{eq:TforXwithM333}
\end{equation}
\end{enumerate}
\end{corollary}

\begin{proof}
Thanks to (\ref{eq:SA-123}), we obtain 
(\ref{eq:TforXwithM444}) and (\ref{eq:TforXwithM333}) 
from 
(\ref{eq:TforXwithM222}) and (\ref{eq:TforXwithM111}), 
respectively.
\end{proof}

\begin{remark}\label{rem:1st-term-123}
Recall $T_{\infty}(w)$ from (\ref{eq:Tinfty123}). 
For $n\in\N\cup\{0\}$, we have $\gamma(n)=\sum_{k=0}^{\infty} \tilde{c}_k \tilde{c}_{n+k}^*$ and 
$\gamma(-n)=\sum_{k=0}^{\infty} \tilde{c}_{n+k} \tilde{c}_{k}^*$ (see (2.13) in \cite{IKP2}), 
hence $T_{\infty}(w) = \tilde{C}_{\infty} (\tilde{C}_{\infty})^*$, 
where
\[
\tilde{C}_{\infty} := 
\left(
\begin{matrix}
\tilde{c}_0     &  \tilde{c}_1  &  \tilde{c}_2  &    \cdots   \cr
                &  \tilde{c}_0  &  \tilde{c}_1  &    \cdots   \cr
                &               &  \tilde{c}_0  &    \cdots   \cr
 \mbox{\huge 0} &               &               &    \ddots   
\end{matrix}
\right).
\]
On the other hand, it follows from $\tilde{h}(z)\tilde{h}(z)^{-1}=I_d$ that 
$\sum_{k=0}^n \tilde{c}_k \tilde{a}_{n-k} = -\delta_{n0}I_d$ for $n\in\N\cup\{0\}$, 
hence $\tilde{C}_{\infty} \tilde{A}_{\infty} = -I_{\infty}$, where
\[
\tilde{A}_{\infty} := 
\left(
\begin{matrix}
\tilde{a}_0     &  \tilde{a}_1  &  \tilde{a}_2  &    \cdots   \cr
                &  \tilde{a}_0  &  \tilde{a}_1  &    \cdots   \cr
                &               &  \tilde{a}_0  &    \cdots   \cr
 \mbox{\huge 0} &               &               &    \ddots   
\end{matrix}
\right),
\qquad 
I_{\infty} := 
\left(
\begin{matrix}
1  &  0  &  0  &   \cdots   \cr
0  &  1  &  0  &   \cdots   \cr
0  &  0  &  1  &   \cdots   \cr
\vdots     &  \vdots      &  \vdots      &   \ddots   
\end{matrix}
\right).
\]
Combining, we have $T_{\infty}(w)^{-1} = (\tilde{A}_{\infty})^* \tilde{A}_{\infty}$. 
Thus, we find that the first term $\sum_{\ell = 1}^{s\wedge t} \tilde{a}_{s - \ell}^* \tilde{a}_{t - \ell}$ 
in (\ref{eq:TforXwithM222}) or (\ref{eq:TforXwithM444}) coincides with the $(s,t)$ block of 
$T_{\infty}(w)^{-1}$.
\end{remark}

For $n\in\N$, we define
\begin{equation}
\tilde{A}_n := 
\left(
\begin{matrix}
\tilde{a}_0    &  \tilde{a}_1  &  \tilde{a}_2  &    \cdots   &  \tilde{a}_{n-1}  \cr
               &  \tilde{a}_0  &  \tilde{a}_1  &    \cdots   &  \tilde{a}_{n-2}  \cr
               &               &  \ddots       &    \ddots   &  \vdots           \cr
               &               &               &    \ddots   &  \tilde{a}_1      \cr
\mbox{\huge 0} &               &               &             &  \tilde{a}_0
\end{matrix}
\right) \in \C^{dn\times dn}
\label{eq:tildeAnU-234}
\end{equation}
and
\begin{equation}
A_n := 
\left(
\begin{matrix}
a_0      &          &          &          &   \mbox{\huge 0} \cr
a_1      &  a_0     &          &          &                  \cr
a_2      &  a_1     &  \ddots  &          &                  \cr
\vdots   &  \vdots  &  \ddots  &  \ddots  &                  \cr
a_{n-1}  &  a_{n-2} &  \cdots  &  a_1     &  a_0
\end{matrix}
\right) \in \C^{dn\times dn}.
\label{eq:AnL-234}
\end{equation}

The next lemma will turn out to be useful in Section \ref{sec:6}.

\begin{lemma}\label{lem:decomp234}
For $n\in\N$ and $s, t \in \{1,\dots,n\}$, we have the following two equalities:
\begin{align*}
\left(\tilde{A}_n^* \tilde{A}_n\right)^{s,t}
&= \sum_{\ell = 1}^{s\wedge t} \tilde{a}_{s - \ell}^* \tilde{a}_{t - \ell},
\\
\left(A_n^* A_n\right)^{s,t}
&= \sum_{\ell=s\vee t}^n a_{\ell - s}^* a_{\ell - t}.
\end{align*}
\end{lemma}

The proof of Lemma \ref{lem:decomp234} is straightforward and will be omitted.


\section{Proof of Theorem \ref{thm:TforXwithM123}}\label{sec:3}

In this section, we prove Theorem \ref{thm:TforXwithM123}. 
We assume (\ref{eq:A}) and (\ref{eq:M}). 
Let $\{X_k\}=\{X_k:k\in\Z\}$ be a $\C^d$-valued, centered,
weakly stationary process, defined on a probability space $(\Omega, \mcF, P)$, 
that has spectral density $w$, hence autocovariance function $\gamma$. 
Thus we have 
$E[X_k X_0^*] = \gamma(k) = \int_{-\pi}^{\pi}e^{-ik\theta}w(e^{i\theta})(d\theta/(2\pi))$ for 
$k\in\Z$.

Write $X_k=(X^1_k,\dots,X^d_k)^{\top}$, and
let $V$ be the complex Hilbert space
spanned by all the entries $\{X^j_k: k\in\Z,\ j=1,\dots,d\}$ in $L^2(\Omega, \mcF, P)$,
which has inner product $(x, y)_{V}:=E[x\overline{y}]$ and
norm $\Vert x\Vert_{V}:=(x,x)_{V}^{1/2}$.
For $J\subset \Z$ such as $\{n\}$,
$(-\infty,n]:=\{n,n-1,\dots\}$, $[n,\infty):=\{n,n+1,\dots\}$,
and $[m,n]:=\{m,\dots,n\}$ with $m\le n$,
we define the closed subspace $V_J^X$ of $V$ by
\[
V_J^X:=\cspn \{X^j_k: j=1,\dots,d,\ k\in J\}.
\]
Let $P_J$ and $P_J^{\perp}$ be the orthogonal projection operators of $V$ onto
$V_J^X$ and $(V_J^X)^{\perp}$, respectively, where 
$(V_J^X)^{\bot}$ denotes the orthogonal complement of $V_J^X$ in $V$.

By Theorem 3.1 in \cite{I00} for $d=1$ and Corollary 3.6 in \cite{IKP1} for general $d\ge 1$, the conditions 
(\ref{eq:A}) and (\ref{eq:M}) imply the following \textit{intersection of past and future}
property:
\begin{equation}
V_{(-\infty,n]}^X\cap V_{[1,\infty)}^X=V_{[1,n]}^X,\qquad n \in \N.
\label{eq:IPF}
\end{equation}

Let $V^d$ be the space of $\C^d$-valued random variables on
$(\Omega, \mcF, P)$ whose entries belong to $V$.
The norm $\Vert x\Vert_{V^d}$ of $x=(x^1,\dots,x^d)^{\top}\in V^d$ is given by
$\Vert x\Vert_{V^d}:=(\sum_{i=1}^d \Vert x^i\Vert_V^2)^{1/2}$.
For $J\subset\mathbb{Z}$ and
$x=(x^1,\dots,x^d)^{\top}\in V^d$, we write $P_Jx$ for $(P_Jx^1, \dots, P_Jx^d)^{\top}$.
We define $P_J^{\perp}x$ in a similar way. 
For $x=(x^1,\dots,x^d)^{\top}$ and $y=(y^1,\dots,y^d)^{\top}$ in $V^d$,
\[
\langle x,y\rangle:=E[xy^*]=((x^k, y^{\ell})_V)_{1\le k, \ell\le d}
\in \C^{d\times d}
\]
stands for the Gram matrix of $x$ and $y$.

Let
\[
X_k=\int_{-\pi}^{\pi}e^{-ik\theta}\eta(d\theta),\qquad k\in\Z,
\]
be the spectral representation of $\{X_k\}$, where $\eta$ is a $\C^d$-valued
random spectral measure. 
We define a $d$-variate stationary process $\{\varepsilon_k:k\in\Z\}$, 
called the \textit{forward innovation
process}\/ of $\{X_k\}$, by
\[
\varepsilon_k:=\int_{-\pi}^{\pi}e^{-ik\theta}h(e^{i\theta})^{-1}\eta(d\theta),\qquad k\in\Z.
\]
Then, $\{\varepsilon_k\}$ satisfies 
$\langle \varepsilon_n, \varepsilon_m\rangle = \delta_{n m}I_d$ and 
$V_{(-\infty,n]}^X=V_{(-\infty,n]}^{\varepsilon}$ for $n\in\Z$,
hence
\[
(V_{(-\infty,n]}^X)^{\bot} = V_{[n+1, \infty)}^{\varepsilon},\qquad n\in\Z.
\]
Recall the outer function $h_{\sharp}$ in $H_2^{d\times d}(\T)$ from (\ref{eq:decomp888}). 
We define the \textit{backward innovation process}\/ $\{\tilde{\varepsilon}_k: k\in\Z\}$ of $\{X_k\}$ by
\[
\tilde{\varepsilon}_k:=\int_{-\pi}^{\pi}e^{ik\theta}\{h_{\sharp}(e^{i\theta})^*\}^{-1}
\eta(d\theta),\qquad k\in\Z.
\]
Then, $\{\tilde{\varepsilon}_k\}$ satisfies
$\langle \tilde{\varepsilon}_n, \tilde{\varepsilon}_m\rangle=\delta_{n m}I_d$ and 
$V_{[-n,\infty)}^X=V_{(-\infty, n]}^{\tilde{\varepsilon}}$ for $n\in\Z$, 
hence
\[
(V_{[-n,\infty)}^X)^{\bot} = V_{[n+1,\infty)}^{\tilde{\varepsilon}},\qquad n\in\Z
\]
(see Section 2 in \cite{IKP2}). Moreover, by Lemma 4.1 in \cite{IKP2}, we have
\begin{equation}
\langle\varepsilon_{\ell}, \tilde{\varepsilon}_{m}\rangle = -\beta_{\ell+m},\qquad
\langle\tilde{\varepsilon}_{m}, \varepsilon_{\ell}\rangle = -\beta_{\ell+m}^*,
\qquad \ell, m \in \Z.
\label{eq:beta234}
\end{equation}

By (\ref{eq:beta234}), for $\{s_{\ell}\}\in \ell_{2+}^{d\times d}$ and $n\in\N$, 
\begin{align}
P_{[1,\infty)}^{\perp}
\left(\sum_{\ell=0}^{\infty} s_{\ell} \varepsilon_{n+1+\ell}\right)
&=-\sum_{\ell=0}^\infty
\left(\sum_{m=0}^\infty s_m \beta_{n+1+\ell+m}\right)
\tilde{\varepsilon}_{\ell},
\label{eq:proj555}\\
P_{(-\infty,n]}^{\perp}
\left(\sum_{\ell=0}^\infty s_{\ell} \tilde{\varepsilon}_{\ell}\right)
&=-\sum_{\ell=0}^\infty
\left(\sum_{m=0}^\infty s_m \beta_{n+1+\ell+m}^*\right)
\varepsilon_{n+1+\ell}.
\label{eq:proj666}
\end{align}
Therefore,
\[
\left\{\sum_{m=0}^\infty s_m \beta_{n+1+\ell+m}\right\}_{\ell=0}^{\infty}, \ 
\left\{\sum_{m=0}^\infty s_m \beta_{n+1+\ell+m}^*\right\}_{\ell=0}^{\infty}\ \in \ 
\ell_{2+}^{d\times d}.
\]
See Lemma 4.2 in \cite{IKP2}. In particular, 
for $n\in\mathbb{N}$, $u \in \{1,\dots,n\}$ and $k\in\mathbb{N}$, we can define 
the sequences $\{b_{n,u,\ell}^k\}_{\ell=0}^{\infty}\in \ell_{2+}^{d\times d}$ 
and $\{\tilde{b}_{n,u,\ell}^k\}_{\ell=0}^{\infty}\in \ell_{2+}^{d\times d}$
by the recursions (\ref{eq:recurs123}) and (\ref{eq:til-recurs123}), respectively.

By (\ref{eq:A}) and (\ref{eq:M}), $\{X_k\}$ has the 
dual process $\{X^{\prime}_k: k\in\Z\}$, which is a $\C^d$-valued, centered, 
weakly stationary process 
characterized by the biorthogonality relation
\[
\langle X_s,X^{\prime}_t\rangle=\delta_{st}I_d, \qquad s, t\in\Z
\]
(see \cite{M60}). 
Recall $\{a_k\} \in \ell_{2+}^{d\times d}$ and $\{\tilde{a}_k\} \in \ell_{2+}^{d\times d}$ 
from (\ref{eq:AR111}) and (\ref{eq:AR222}), respectively. 
The dual process $\{X^{\prime}_k\}$ admits the following two MA representations (see Section 5 in \cite{IKP2}):
\begin{align}
X^{\prime}_n&=-\sum_{\ell=n}^\infty a_{\ell-n}^* \varepsilon_{\ell}, \qquad n\in\Z, 
\label{eq:MAdual111}\\
X^{\prime}_n&=-\sum_{\ell=-n}^{\infty} \tilde{a}_{\ell+n}^* \tilde{\varepsilon}_{\ell},\qquad n\in\Z.
\label{eq:MAdual222}
\end{align}

The next theorem is the key to the proof of Theorem \ref{thm:TforXwithM123}.

\begin{theorem}\label{thm:Tdual123}
Assume (\ref{eq:A}) and (\ref{eq:M}). Then, for $n\in\N$ and $s,t \in \{1,\dots,n\}$, we have (\ref{eq:Tdual123}).
\end{theorem}

\begin{proof}
Fix $n\in\N$. For $s\in \{1,\dots,n\}$, we can write 
$P_{[1,n]}X^{\prime}_s = \sum_{k=1}^n q_{s,k} X_k$ 
for some $q_{s,k}\in\C^{d\times d}$, $k\in\{1,\dots,n\}$. 
For $s,t \in \{1,\dots,n\}$, we have
\[
\begin{aligned}
\delta_{st}I_d 
&= \langle X^{\prime}_s, X_t\rangle 
= \langle X^{\prime}_s, P_{[1,n]}X_t\rangle 
= \langle P_{[1,n]}X^{\prime}_s, X_t \rangle
= \left\langle \sum_{k=1}^n q_{s,k} X_k, X_t\right\rangle\\
&=\sum_{k=1}^n q_{s,k} \left\langle X_k, X_t \right\rangle 
=\sum_{k=1}^n q_{s,k} \gamma(k-t),
\end{aligned}
\]
or $Q_n T_n(w) = I_{dn}$, where 
$Q_n := (q_{s,k})_{1\le s, k\le n}\in \C^{dn\times dn}$. Therefore, we have 
$Q_n = T_n(w)^{-1}$. 
However, 
\[
\langle X^{\prime}_s, P_{[1,n]}X^{\prime}_t\rangle 
= \langle P_{[1,n]}X^{\prime}_s, X^{\prime}_t\rangle 
= \left\langle \sum_{k=1}^n q_{s,k} X_k, X^{\prime}_t \right\rangle
= \sum_{k=1}^n q_{s,k} \langle X_k, X^{\prime}_t\rangle = q_{s,t}.
\]
Thus, the theorem follows.
\end{proof}

\begin{lemma}\label{lem:Tdual456}
Assume (\ref{eq:A}) and (\ref{eq:M}). Then, 
for $n\in\mathbb{N}$ and $s,t \in \{1,\dots,n\}$, the following two equalities hold:
\begin{align}
\langle X^{\prime}_s, P_{[1,n]}X^{\prime}_t\rangle 
&= \sum_{\ell=s\vee t}^n a_{\ell-s}^* a_{\ell-t} 
+ \sum_{u=t}^n \langle X^{\prime}_s, P_{[1,n]}^{\bot} \varepsilon_u\rangle a_{u-t},
\label{eq:dual777}\\
\langle X^{\prime}_s, P_{[1,n]}X^{\prime}_t\rangle 
&= \sum_{\ell=1}^{s\wedge t} \tilde{a}_{s-\ell}^* \tilde{a}_{t-\ell} 
+ \sum_{u=1}^t \langle X^{\prime}_s, P_{[1,n]}^{\bot}\tilde{\varepsilon}_{-u}\rangle \tilde{a}_{t-u}.
\label{eq:dual888}
\end{align}
\end{lemma}

\begin{proof}
First, we prove (\ref{eq:dual777}). Since 
$V_{[1,n]}^{X} \subset V_{(-\infty,n]}^{X}$, we have
\[
\begin{aligned}
\langle X^{\prime}_s, P_{[1,n]} X^{\prime}_t\rangle
&=\langle X^{\prime}_s, P_{[1,n]} P_{(-\infty,n]} X^{\prime}_t\rangle\\
&=\langle X^{\prime}_s, P_{(-\infty,n]} X^{\prime}_t\rangle 
- \langle X^{\prime}_s, P_{[1,n]}^{\bot} P_{(-\infty,n]} X^{\prime}_t\rangle.
\end{aligned}
\]
On the other hand, from (\ref{eq:MAdual111}), we have 
$P_{(-\infty,n]} X^{\prime}_t = -\sum_{m=t}^n a_{m-t}^* \varepsilon_{m}$, 
hence
\[
\langle X^{\prime}_s, P_{(-\infty,n]} X^{\prime}_t\rangle 
= \left\langle \sum_{\ell=s}^{\infty} a_{\ell-s}^* \varepsilon_{\ell}, 
\sum_{m=t}^n a_{m-t}^* \varepsilon_{m}\right\rangle
=\sum_{\ell=s\vee t}^n a_{\ell-s}^* a_{\ell-t},
\]
and $\langle X^{\prime}_s, P_{[1,n]}^{\bot} P_{(-\infty,n]} X^{\prime}_t\rangle$ 
is equal to
\[
-\left\langle X^{\prime}_s, P_{[1,n]}^{\bot}\left(\sum_{u=t}^n a_{u-t}^* \varepsilon_{u}\right)\right\rangle
=-\sum_{u=t}^n \langle X^{\prime}_s, P_{[1,n]}^{\bot}\varepsilon_{u}\rangle a_{u-t}.
\]
Combining, we obtain (\ref{eq:dual777}).

Next, we prove (\ref{eq:dual888}). Since 
$V_{[1,n]}^{X} \subset V_{[1,\infty)}^{X}$, we have
\[
\begin{aligned}
\langle X^{\prime}_s, P_{[1,n]} X^{\prime}_t\rangle
&=\langle X^{\prime}_s, P_{[1,n]} P_{[1,\infty)} X^{\prime}_t\rangle\\
&=\langle X^{\prime}_s, P_{[1,\infty)} X^{\prime}_t\rangle 
- \langle X^{\prime}_s, P_{[1,n]}^{\bot} P_{[1,\infty)} X^{\prime}_t\rangle.
\end{aligned}
\]
On the other hand, from (\ref{eq:MAdual222}), we have 
$P_{[1, \infty)} X^{\prime}_t 
= -\sum_{m=1}^{t} \tilde{a}_{t-m}^* \tilde{\varepsilon}_{-m}$, 
hence
\[
\langle X^{\prime}_s, P_{[1,\infty)} X^{\prime}_t\rangle 
= \left\langle \sum_{\ell=-\infty}^{s} \tilde{a}_{s-\ell}^* \tilde{\varepsilon}_{-\ell}, 
\sum_{m=1}^{t} \tilde{a}_{t-m}^* \tilde{\varepsilon}_{-m} \right\rangle
=\sum_{\ell=1}^{s\wedge t} \tilde{a}_{s-\ell}^* \tilde{a}_{t-\ell},
\]
and $\langle X^{\prime}_s, P_{[1,n]}^{\bot} P_{[1,\infty)} X^{\prime}_t\rangle$ 
is equal to
\[
-\left\langle X^{\prime}_s, P_{[1,n]}^{\bot}\left( \sum_{u=1}^{t} \tilde{a}_{t-u}^* \tilde{\varepsilon}_{-u} \right)\right\rangle
=-\sum_{u=1}^t \langle X^{\prime}_s, P_{[1,n]}^{\bot}\tilde{\varepsilon}_{-u}\rangle \tilde{a}_{t-u}.
\]
Combining, we obtain (\ref{eq:dual888}).
\end{proof}

For $n\in\mathbb{N}$ and $u \in \{1,\dots,n\}$, 
we define the sequence $\{W_{n,u}^k\}_{k=1}^{\infty}$ in $V^d$ by
\begin{align*}
W_{n,u}^{2k-1}
&= - P_{[1, \infty)}^{\perp}
(P_{(-\infty, n]}^\perp P_{[1, \infty)}^{\perp})^{k-1} \varepsilon_u,\qquad k \in \N,
\\
W_{n,u}^{2k}
&=
(P_{(-\infty,n]}^{\perp} P_{[1, \infty)}^{\perp})^{k} \varepsilon_u,\qquad k \in \N.
\end{align*}

\begin{lemma}\label{lem:basic-rep123}
We assume (\ref{eq:A}) and (\ref{eq:M}). Then, for $n\in\mathbb{N}$ and $u\in\{1,\dots,n\}$, we have
\begin{equation}
P_{[1,n]}^{\perp} \varepsilon_u = -\sum_{k=1}^{\infty} W_{n,u}^{k},
\label{eq:basic111}
\end{equation}
the sum converging strongly in $V^d$.
\end{lemma}

\begin{proof}
Since $\varepsilon_u$ is in $V_{(-\infty,n]}^X$, (\ref{eq:basic111}) follows
from (\ref{eq:IPF}) and Theorem 3.2 in \cite{IKP2}. 
\end{proof}

\begin{proposition}\label{prop:forward638}
We assume (\ref{eq:A}) and (\ref{eq:M}). Then, for $n\in\mathbb{N}$, $u\in\{1,\dots,n\}$ and $k\in\mathbb{N}$, we have
\begin{align}
W_{n,u}^{2k-1}
&=\sum_{\ell=0}^\infty b_{n,u,\ell}^{2k-1} \tilde{\varepsilon}_{\ell},
\label{eq:expan111}\\
W_{n,u}^{2k}
&=\sum_{\ell=0}^\infty b_{n,u,\ell}^{2k} \varepsilon_{n+1+\ell}.
\label{eq:expan222}
\end{align}
\end{proposition}

\begin{proof}
Note that, from the definition of $W_{n,u}^k$,
\[
W_{n,u}^{2k+1}=-P_{[1,\infty)}^\perp W_{n,u}^{2k},
\qquad
W_{n,u}^{2k+2}=-P_{(-\infty,n]}^\perp W_{n,u}^{2k+1}.
\]
We prove (\ref{eq:expan111}) and (\ref{eq:expan222}) by induction. 
First, by (\ref{eq:beta234}), we have
\[
W_{n,u}^1=-P_{[1, \infty)}^{\perp} \varepsilon_u 
= -\sum_{\ell=0}^{\infty} \langle \varepsilon_u, \tilde{\varepsilon}_{\ell}\rangle \tilde{\varepsilon}_{\ell}
=\sum_{\ell=0}^\infty \beta_{u+\ell} \tilde{\varepsilon}_{\ell}
=
\sum_{\ell=0}^\infty b_{n,u,\ell}^1 \tilde{\varepsilon}_{\ell}.
\]
For $k \in \N$, assume that
$W_{n,u}^{2k-1}=\sum_{\ell=0}^\infty b_{n,u,\ell}^{2k-1} \tilde{\varepsilon}_{\ell}$.
Then, by (\ref{eq:proj666}),
\[
\begin{aligned}
W_{n,u}^{2k}
&=-P_{(-\infty,n]}^{\perp}
\left(\sum_{\ell=0}^{\infty} b_{n,u,\ell}^{2k-1} \tilde{\varepsilon}_{\ell}\right)
=\sum_{\ell=0}^{\infty}
\left(\sum_{m=0}^\infty b_{n,u,m}^{2k-1} \beta_{n+1+m+\ell}^*
\right) \varepsilon_{n+1+\ell}\\
&=\sum_{\ell=0}^\infty b_{n,u,\ell}^{2k} \varepsilon_{n+1+\ell},
\end{aligned}
\]
and, by (\ref{eq:proj555}),
\[
\begin{aligned}
W_{n,u}^{2k+1}
&=-P_{[1, \infty)}^{\perp}
\left(\sum_{\ell=0}^\infty b_{n,u,\ell}^{2k} \varepsilon_{n+1+\ell}\right)
=\sum_{\ell=0}^{\infty}
\left(\sum_{m=0}^{\infty} b_{n,u,m}^{2k} \beta_{n+1+m+\ell}
\right) \tilde{\varepsilon}_{\ell}\\
&=\sum_{\ell=0}^{\infty} b_{n,u,\ell}^{2k+1} \tilde{\varepsilon}_{\ell}.
\end{aligned}
\]
Thus (\ref{eq:expan111}) and (\ref{eq:expan222}) follow.
\end{proof}

For $n\in\mathbb{N}$ and $u \in \{1,\dots,n\}$, 
we define the sequence $\{\tilde{W}_{n,u}^k\}_{k=1}^{\infty}$ in
$V^d$ by
\begin{align*}
\tilde{W}_{n,u}^{2k-1}
&= - P_{(-\infty, n]}^\perp
(P_{[1, \infty)}^{\perp} P_{(-\infty, n]}^\perp)^{k-1} \tilde{\varepsilon}_{-u},
\qquad k \in \N,
\\
\tilde{W}_{n,u}^{2k}
&=
(P_{[1, \infty)}^{\perp} P_{(-\infty,n]}^{\perp})^{k} \tilde{\varepsilon}_{-u},
\qquad k \in \N.
\end{align*}

\begin{lemma}\label{lem:basic-rep456}
We assume (\ref{eq:A}) and (\ref{eq:M}). Then, for $n\in\mathbb{N}$ and $u\in\{1,\dots,n\}$, we have
\begin{equation}
P_{[1,n]}^{\perp} \tilde{\varepsilon}_{-u} = -\sum_{k=1}^{\infty} \tilde{W}_{n,u}^{k},
\label{eq:basic222}
\end{equation}
the sum converging strongly in $V^d$.
\end{lemma}

\begin{proof}
Since $\tilde{\varepsilon}_{-u}$ is in $V_{[1,\infty)}^X$, (\ref{eq:basic222}) follows
from (\ref{eq:IPF}) and Theorem 3.2 in \cite{IKP2}. 
\end{proof}

\begin{proposition}\label{prop:til-forward638}
We assume (\ref{eq:A}) and (\ref{eq:M}). 
Then, for $n\in\mathbb{N}$, $u\in\{1,\dots,n\}$ and $k\in\mathbb{N}$, we have
\begin{align}
\tilde{W}_{n,u}^{2k-1}
&=\sum_{\ell=0}^\infty \tilde{b}_{n,u,\ell}^{2k-1} \varepsilon_{n+1+\ell},
\label{eq:til-expan111}\\
\tilde{W}_{n,u}^{2k}
&=\sum_{\ell=0}^\infty \tilde{b}_{n,u,\ell}^{2k} \tilde{\varepsilon}_{\ell}.
\label{eq:til-expan222}
\end{align}
\end{proposition}

\begin{proof}
Note that, from the definition of $\tilde{W}_{n,u}^k$,
\[
\tilde{W}_{n,u}^{2k+1}=-P_{(-\infty,n]}^\perp \tilde{W}_{n,u}^{2k},
\qquad
\tilde{W}_{n,u}^{2k+2}=- P_{[1,\infty)}^\perp \tilde{W}_{n,u}^{2k+1}.
\]
We prove (\ref{eq:til-expan111}) and (\ref{eq:til-expan222}) by induction. 
First, by (\ref{eq:beta234}), we have
\[
\begin{aligned}
\tilde{W}_{n,u}^1&= - P_{(-\infty,n]}^\perp \tilde{\varepsilon}_{-u} 
= -\sum_{\ell=0}^{\infty} \langle \tilde{\varepsilon}_{-u}, \varepsilon_{n+1+\ell}\rangle \varepsilon_{n+1+\ell}\\
&=\sum_{\ell=0}^\infty \beta_{n+1-u+\ell}^* \varepsilon_{n+1+\ell}
=
\sum_{\ell=0}^\infty \tilde{b}_{n,u,\ell}^1 \varepsilon_{n+1+\ell}.
\end{aligned}
\]
For $k \in \N$, assume that
$\tilde{W}_{n,u}^{2k-1}=\sum_{\ell=0}^\infty \tilde{b}_{n,u,\ell}^{2k-1} \varepsilon_{n+1+\ell}$.
Then, by (\ref{eq:proj555}),
\[
\begin{aligned}
\tilde{W}_{n,u}^{2k}
&=-P_{[1,\infty)}^\perp
\left(\sum_{\ell=0}^{\infty} \tilde{b}_{n,u,\ell}^{2k-1} \varepsilon_{n+1+\ell}\right)
=\sum_{\ell=0}^{\infty}
\left(\sum_{m=0}^\infty \tilde{b}_{n,u,m}^{2k-1} \beta_{n+1+m+\ell}
\right) \tilde{\varepsilon}_l\\
&=\sum_{\ell=0}^\infty \tilde{b}_{n,u,\ell}^{2k} \tilde{\varepsilon}_{\ell},
\end{aligned}
\]
and, by (\ref{eq:proj666}),
\[
\begin{aligned}
\tilde{W}_{n,u}^{2k+1}
&=-P_{(\infty,n]}^{\perp}
\left(\sum_{\ell=0}^\infty \tilde{b}_{n,u,\ell}^{2k} \tilde{\varepsilon}_{\ell}\right)
=\sum_{\ell=0}^{\infty}
\left(\sum_{m=0}^{\infty} \tilde{b}_{n,u,m}^{2k} \beta_{n+1+m+\ell}^*
\right) \varepsilon_{n+1+\ell}\\
&=\sum_{\ell=0}^{\infty} \tilde{b}_{n,u,\ell}^{2k+1} \varepsilon_{n+1+\ell}.
\end{aligned}
\]
Thus (\ref{eq:til-expan111}) and (\ref{eq:til-expan222}) follow.
\end{proof}

We are ready to prove Theorem \ref{thm:TforXwithM123}.

\begin{proof}
(i)\ For $n\in\N$, $s, u \in \{1,\dots,n\}$ and $k\in\N$, 
we see from (\ref{eq:MAdual111}) and (\ref{eq:til-expan111}) that
\[
\langle X^{\prime}_s, \tilde{W}_{n,u}^{2k-1} \rangle
= - \sum_{\ell=0}^{\infty} a_{n+1-s+\ell}^* (\tilde{b}_{n,u,\ell}^{2k-1})^*,
\]
and from (\ref{eq:MAdual222}) and (\ref{eq:til-expan222}) that
\[
\langle X^{\prime}_s, \tilde{W}_{n,u}^{2k} \rangle
= - \sum_{\ell=0}^{\infty} \tilde{a}_{s+\ell}^* (\tilde{b}_{n,u,\ell}^{2k})^*.
\]
Therefore, by Lemma \ref{lem:basic-rep456}, 
$\langle X^{\prime}_s, P_{[1,n]}^{\perp} \tilde{\varepsilon}_{-u} \rangle$ is equal to
\[
- \sum_{k=1}^{\infty} \langle X^{\prime}_s, \tilde{W}_{n,u}^{k} \rangle
= \sum_{k=1}^{\infty} 
\left\{  
\sum_{\ell=0}^{\infty} \tilde{b}_{n,u,\ell}^{2k-1} a_{n+1-s+\ell}
+
\sum_{\ell=0}^{\infty} \tilde{b}_{n,u,\ell}^{2k}   \tilde{a}_{s+\ell} 
\right\}^*.
\]
The assertion (i) follows from this, Theorem \ref{thm:Tdual123} and Lemma \ref{lem:Tdual456}.

(ii)\ For $n\in\N$, $s, u \in \{1,\dots,n\}$ and $k\in\N$, 
we see from (\ref{eq:MAdual222}) and (\ref{eq:expan111}) that
\[
\langle X^{\prime}_s, W_{n,u}^{2k-1} \rangle
= - \sum_{\ell=0}^{\infty} \tilde{a}_{s+\ell}^* (b_{n,u,\ell}^{2k-1})^*,
\]
and from (\ref{eq:MAdual111}) and (\ref{eq:expan222}) that
\[
\langle X^{\prime}_s, W_{n,u}^{2k} \rangle
= - \sum_{\ell=0}^{\infty} a_{n+1-s+\ell}^* (b_{n,u,\ell}^{2k})^*.
\]
Therefore, by Lemma \ref{lem:basic-rep123}, 
$\langle X^{\prime}_s, P_{[1,n]}^{\perp} \varepsilon_u \rangle$ is equal to 
\[
- \sum_{k=1}^{\infty} \langle X^{\prime}_s, W_{n,u}^{k} \rangle
= \sum_{k=1}^{\infty} 
\left\{  
\sum_{\ell=0}^{\infty} b_{n,u,\ell}^{2k-1} \tilde{a}_{s+\ell}
+
\sum_{\ell=0}^{\infty} b_{n,u,\ell}^{2k}   a_{n+1-s+\ell} 
\right\}^*.
\]
The assertion (ii) follows from this, Theorem \ref{thm:Tdual123} and Lemma \ref{lem:Tdual456}.
\end{proof}


\section{Strong convergence result for Toeplitz systems}\label{sec:4}

In this section, we use Theorem \ref{thm:TforXwithM123} to show a strong convergence 
result for solutions of block Toeplitz systems. 
We assume (\ref{eq:A}) and (\ref{eq:S}). 
Then $w$ is continuous on $\T$ since 
$w(e^{i\theta})=(2\pi)^{-1}\sum_{k\in\Z} e^{ik\theta} \gamma(k)$. 
In particular, (\ref{eq:M}) is also satisfied. 
The conditions (\ref{eq:A}) and (\ref{eq:S}) also imply 
that all of $\{a_k\}$, $\{c_k\}$, $\{\tilde{a}_k\}$ and 
$\{\tilde{c}_k\}$ belong to $\ell_{1+}^{d\times d}$. 
See Theorem 3.3 and (3.3) in \cite{KB18}; see also Theorem 4.1 in \cite{I08}.
In particular, we have 
$h(e^{i\theta})^{-1} = - \sum_{k=0}^{\infty} e^{ik\theta} a_k$ and
$h_{\sharp}(e^{i\theta}) = \tilde{h}(e^{-i\theta})^* = \sum_{k=0}^{\infty} e^{ik\theta} \tilde{c}_k^*$, 
hence, by (\ref{eq:beta-def667}), 
\begin{equation}
\beta_k = \sum_{j=0}^{\infty} a_{j+k} \tilde{c}_j,\qquad k \in \N\cup\{0\}.
\label{eq:beta-a-c111}
\end{equation}

Under (\ref{eq:A}) and (\ref{eq:S}), we define
\[
F(n):=\left(\sum_{j=0}^{\infty}\Vert \tilde{c}_j\Vert\right)
\sum_{\ell=n}^{\infty}\Vert a_{\ell}\Vert, \qquad 
n \in \N\cup\{0\}.
\]
Then $F(n)$ decreases to zero as $n\to\infty$.

We need the next lemma in the proof of Theorem \ref{thm:Bax-conv123} below.

\begin{lemma}\label{lem:estim4.2}
Assume (\ref{eq:A}) and (\ref{eq:S}). Then, for $n, k\in\N$ and $u\in\{1,\dots,n\}$, we have
\begin{equation}
\sum_{\ell=0}^{\infty}
\Vert \tilde{b}^{k}_{n,u,\ell}\Vert\le F(n+1)^{k-1}F(n+1-u).
\label{eq:ac-ineq985}
\end{equation}
\end{lemma}

\begin{proof}
For $m\in\N$, 
we see from (\ref{eq:beta-a-c111}) that
\[
\sum_{\ell=0}^{\infty} \Vert \beta_{m+\ell}\Vert
\le \sum_{j=0}^{\infty} \Vert \tilde{c}_j\Vert \sum_{\ell=0}^{\infty} \Vert a_{m+j+\ell}\Vert
\le \sum_{j=0}^{\infty} \Vert \tilde{c}_j\Vert \sum_{\ell=m}^{\infty} \Vert a_{\ell}\Vert,
\]
hence
\begin{equation}
\sum_{\ell=0}^{\infty} \Vert \beta_{m+\ell}\Vert \le F(m).
\label{eq:beta-ineq519}
\end{equation}

Let $n\in\N$ and $u\in\{1,\dots,n\}$. We use induction on $k$ to prove (\ref{eq:ac-ineq985}). 
Since $\tilde{b}^1_{n,u,\ell}=\beta_{n+1-u+\ell}^*$, we see from (\ref{eq:beta-ineq519}) that
\[
\sum_{\ell=0}^{\infty} \Vert \tilde{b}^1_{n,u,\ell} \Vert
=\sum_{\ell=0}^{\infty}\Vert \beta_{n+1-u+\ell} \Vert
\le F(n+1-u).
\]
We assume (\ref{eq:ac-ineq985}) for $k\in\N$. 
Then, again by (\ref{eq:beta-ineq519}),
\[
\begin{aligned}
\sum_{\ell=0}^{\infty}
\Vert \tilde{b}^{k+1}_{n,u,\ell} \Vert
&\le \sum_{m=0}^{\infty} \Vert \tilde{b}^{k}_{n,u,m} \Vert 
\sum_{\ell=0}^{\infty} \Vert \beta_{n + 1 + m + \ell}\Vert\\
&\le F(n+1) \sum_{m=0}^{\infty} \Vert \tilde{b}^k_{n,u,m} \Vert \le F(n+1)^{k}F(n+1-u).
\end{aligned}
\]
Thus (\ref{eq:ac-ineq985}) with $k$ replaced by $k+1$ also holds.
\end{proof}

For $\{y_k\}_{k=1}^{\infty} \in \ell_1^{d\times d}(\N)$, the solution 
$Z_{\infty}$ to (\ref{eq:TSinfty183}) with (\ref{eq:Tinfty123}) and (\ref{eq:Rrinfty-314}) is 
given by (\ref{eq:Zzinfty-314}) with
\begin{equation}
z_s = \sum_{t=1}^{\infty} \sum_{\ell = 1}^{s\wedge t} 
\tilde{a}_{s - \ell}^* \tilde{a}_{t - \ell} y_t \in \C^{d\times d}, \qquad s\in\N
\label{eq:z-sequence123}
\end{equation}
(see Remark \ref{rem:1st-term-123} in Section \ref{sec:2}). 
Notice that the sum in (\ref{eq:z-sequence123}) converges absolutely.

\begin{theorem}\label{thm:Bax-conv123}
We assume (\ref{eq:A}) and (\ref{eq:S}). Let $\{y_k\}_{k=1}^{\infty} \in \ell_1^{d\times d}(\N)$. 
Then, for 
$Z_n$ in (\ref{eq:Zz-314})--(\ref{eq:Rr-314}) and 
$Z_{\infty}$ in (\ref{eq:Zzinfty-314})--(\ref{eq:Rrinfty-314}), 
we have (\ref{eq:conv-z-234}).
\end{theorem}

\begin{proof}
By Theorem \ref{thm:TforXwithM123} (i), we have
\[
\begin{aligned}
z_{n,s} &= \sum_{t=1}^{n} \sum_{\ell = 1}^{s\wedge t} 
\tilde{a}_{s - \ell}^* \tilde{a}_{t - \ell} y_t 
+ \sum_{t=1}^{n} \sum_{u=1}^t 
          \sum_{\ell = 0}^{\infty} a_{n + 1 - s + \ell}^* \beta_{n+1-u+\ell} \tilde{a}_{t-u} y_t\\
&\quad\quad + \sum_{t=1}^{n} \sum_{u=1}^t \sum_{k=1}^{\infty} 
     \left\{ 
          \sum_{\ell = 0}^{\infty} \tilde{b}_{n,u,\ell}^{2k+1} a_{n + 1 - s + \ell} 
       +  \sum_{\ell = 0}^{\infty} \tilde{b}_{n,u,\ell}^{2k}   \tilde{a}_{s + \ell}
     \right\}^* \tilde{a}_{t-u} y_t,
\end{aligned}
\]
hence, by (\ref{eq:z-sequence123}),
$\sum_{s=1}^n \Vert z_{n,s} - z_s\Vert 
\le S_1(n) + S_2(n) + S_3(n) + S_4(n)$, 
where
\begin{align*}
S_1(n) &:= \sum_{t=n+1}^{\infty} \sum_{s=1}^n \sum_{\ell = 1}^{s} 
\Vert \tilde{a}_{s - \ell}\Vert \Vert \tilde{a}_{t - \ell}\Vert \Vert y_t\Vert,\\
S_2(n) &:= \sum_{s=1}^n \sum_{t=1}^{n} \sum_{u=1}^t \sum_{\ell = 0}^{\infty} 
\Vert a_{n + 1 - s + \ell}\Vert \Vert \beta_{n+1-u+\ell}\Vert \Vert \tilde{a}_{t-u}\Vert 
\Vert y_t\Vert,\\
S_3(n) &:=  \sum_{s=1}^n \sum_{t=1}^{n} \sum_{u=1}^t \sum_{k=1}^{\infty} \sum_{\ell = 0}^{\infty} 
\Vert \tilde{b}_{n,u,\ell}^{2k+1}\Vert \Vert a_{n + 1 - s + \ell}\Vert 
\Vert \tilde{a}_{t-u}\Vert \Vert y_t\Vert
\end{align*}
and
\[
S_4(n) = \sum_{s=1}^n \sum_{t=1}^{n} \sum_{u=1}^t \sum_{k=1}^{\infty} \sum_{\ell = 0}^{\infty} 
\Vert \tilde{b}_{n,u,\ell}^{2k}\Vert \Vert \tilde{a}_{s + \ell}\Vert 
\Vert \tilde{a}_{t-u}\Vert \Vert y_t\Vert.
\]

By the change of variables $m=s-\ell+1$, we have
\[
\begin{aligned}
S_1(n) 
&= \sum_{t=n+1}^{\infty} \sum_{s=1}^n \sum_{m = 1}^{s} 
\Vert \tilde{a}_{m - 1}\Vert \Vert \tilde{a}_{t + m -s - 1}\Vert \Vert y_t\Vert\\
&=\sum_{t=n+1}^{\infty} \Vert y_t\Vert \sum_{m = 1}^{n} \Vert \tilde{a}_{m - 1}\Vert \sum_{s=m}^n 
\Vert \tilde{a}_{t + m -s - 1}\Vert\\
&\le \left( \sum_{k = 0}^{\infty} \Vert \tilde{a}_{k}\Vert\right)^2 \sum_{t=n+1}^{\infty} \Vert y_t\Vert 
\quad \to \quad 0,\qquad n\to\infty.
\end{aligned}
\]

By (\ref{eq:ac-ineq985}) with $k=1$ or (\ref{eq:beta-ineq519}), we have
\[
\begin{aligned}
S_2(n) &= \sum_{t=1}^{n} \sum_{u=1}^t 
\Vert \tilde{a}_{t-u}\Vert  \Vert y_t\Vert 
\sum_{\ell = 0}^{\infty} \Vert \beta_{n+1-u+\ell}\Vert 
\sum_{s=1}^n \Vert a_{n + 1 - s + \ell} \Vert\\
&\le \left( \sum_{s=1}^{\infty} \Vert a_{s}\Vert\right) 
\sum_{t=1}^{n} \sum_{u=1}^t 
\Vert \tilde{a}_{t-u}\Vert  \Vert y_t\Vert F(n+1-u).
\end{aligned}
\]
Furthermore, by the change of variables $v=t-u+1$, we obtain
\[
\begin{aligned}
\sum_{t=1}^{n} \sum_{u=1}^t 
\Vert \tilde{a}_{t-u}\Vert  \Vert y_t\Vert F(n+1-u)
&=\sum_{t=1}^{\infty} \sum_{u=1}^t 
\Vert \tilde{a}_{t-u}\Vert  \Vert y_t\Vert 1_{[0,n]}(t)F(n+1-u)\\
&=\sum_{t=1}^{\infty} \sum_{v=1}^t 
\Vert \tilde{a}_{v-1}\Vert  \Vert y_t\Vert 1_{[0,n]}(t)F(n-t+v)\\
&\le \sum_{t=1}^{\infty} \sum_{v=1}^{\infty} 
\Vert \tilde{a}_{v-1}\Vert  \Vert y_t\Vert 1_{[0,n]}(t)F(n-t+v).
\end{aligned}
\]
Since
\begin{gather*}
\lim_{n\to\infty} \Vert \tilde{a}_{v-1}\Vert  \Vert y_t\Vert 1_{[0,n]}(t)F(n-t+v)
=0,\qquad t, v\in\N,\\
\Vert \tilde{a}_{v-1}\Vert  \Vert y_t\Vert 1_{[0,n]}(t)F(n-t+v)
\le F(1)\Vert \tilde{a}_{v-1}\Vert  \Vert y_t\Vert,\qquad t, v\in\N,\\
\sum_{t=1}^{\infty} \sum_{v=1}^{\infty} 
\Vert \tilde{a}_{v-1}\Vert  \Vert y_t\Vert <\infty,
\end{gather*}
the dominated convergence theorem yields
\[
\lim_{n\to\infty} \sum_{t=1}^{\infty} \sum_{v=1}^{\infty} 
\Vert \tilde{a}_{v-1}\Vert  \Vert y_t\Vert 1_{[0,n]}(t)F(n-t+v) = 0,
\]
hence $\lim_{n\to\infty} S_2(n) = 0$.

Choose $N\in\N$ such that $F(N+1)<1$. Then, by Lemma \ref{lem:estim4.2}, we have, 
for $n\ge N$,
\[
\begin{aligned}
S_3(n)&=\sum_{t=1}^{n} \sum_{u=1}^t 
\Vert \tilde{a}_{t-u}\Vert \Vert y_t\Vert 
\sum_{k=1}^{\infty} \sum_{\ell = 0}^{\infty} \Vert \tilde{b}_{n,u,\ell}^{2k+1}\Vert 
\sum_{s=1}^n \Vert a_{n + 1 - s + \ell}\Vert\\
&\le F(1)\left( \sum_{s=1}^{\infty} \Vert a_{s}\Vert\right)
\sum_{t=1}^{n} \Vert y_t\Vert \sum_{u=1}^t 
\Vert \tilde{a}_{t-u}\Vert 
\sum_{k=1}^{\infty} F(n+1)^{2k}\\
&\le F(1)\left( \sum_{s=1}^{\infty} \Vert a_{s}\Vert\right)
\left( \sum_{u=0}^{\infty} \Vert \tilde{a}_{u}\Vert\right)
\left(\sum_{t=1}^{\infty} \Vert y_t\Vert\right) 
\frac{F(n+1)^2}{1-F(n+1)^2}.
\end{aligned}
\]
Thus $\lim_{n\to\infty}S_3(n)=0$. Similarly, we have, for $n\ge N$,
\[
S_4(n)
\le F(1)
\left( \sum_{s=0}^{\infty} \Vert \tilde{a}_{s}\Vert\right)^2
\left(\sum_{t=1}^{\infty} \Vert y_t\Vert\right) 
\frac{F(n+1)}{1-F(n+1)^2},
\]
hence $\lim_{n\to\infty}S_4(n)=0$.

Combining, we obtain (\ref{eq:conv-z-234}).
\end{proof}


\section{Closed-form formulas}\label{sec:5}

In this section, we use Theorem \ref{thm:TforXwithM123} to derive 
closed-form formulas for $T_n(w)^{-1}$ with rational symbol $w$ that corresponds to a 
$d$-variate ARMA process.
We assume that the symbol $w$ of $T_n(w)$ is of the form (\ref{eq:farima529}) 
with $h:\T\to \C^{d\times d}$ satisfying (\ref{eq:C}).
Then $h$ is an outer function in $H_2^{d\times d}(\T)$, and 
another outer function $h_{\sharp}\in H_2^{d\times d}(\T)$ that appears in (\ref{eq:decomp888}) 
also satisfies (\ref{eq:C}); see Section 6.2 in \cite{IKP2}. 
Notice that (\ref{eq:farima529}) with (\ref{eq:C}) implies (\ref{eq:A}) and (\ref{eq:M}).

We can write $h(z)^{-1}$ in the form
\begin{equation}
h(z)^{-1} = - \rho_{0,0} - \sum_{{\mu}=1}^{K} \sum_{j=1}^{m_{\mu}} \frac{1}{(1-\barp_{\mu}z)^j}\rho_{{\mu}, j} 
- \sum_{j=1}^{m_0} z^j \rho_{0,j},
\label{eq:hinverse162}
\end{equation}
where
\begin{equation}
\left\{
\begin{aligned}
&K\in\N\cup\{0\}, \qquad m_{\mu}\in\N, \quad \mu \in \{1,\dots,K\}, \qquad m_0\in\N\cup\{0\},\\
&p_{\mu}\in\D\setminus \{0\}, \quad \mu \in \{1,\dots,K\}, \qquad p_{\mu}\neq p_{\nu}, \quad \mu\neq \nu,\\
&\rho_{{\mu}, j}\in\C^{d\times d}, \quad \mu \in \{0,\dots,K\},\ j \in \{1,\dots,m_{\mu}\}, 
\qquad \rho_{0,0} \in\C^{d\times d},\\
&\rho_{{\mu},m_{\mu}}\neq 0, \quad \mu \in \{1,\dots,K\},\\
&\rho_{0,m_0}\neq 0\quad \mbox{if}\ \ m_0\ge 1.
\end{aligned}
\right.
\label{eq:hinverse163}
\end{equation}
Here the convention $\sum_{k=1}^0=0$ is adopted in the sums on the right-hand side of (\ref{eq:hinverse162}).
For example, if $m_0=0$, then
\[
h(z)^{-1} = - \rho_{0,0} - \sum_{{\mu}=1}^{K} \sum_{j=1}^{m_{\mu}} \frac{1}{(1-\barp_{\mu}z)^j}\rho_{{\mu}, j},
\]
while, if $K=0$, then
\begin{equation}
h(z)^{-1} = - \rho_{0,0} - \sum_{j=1}^{m_0} z^j \rho_{0,j}
\label{eq:K000}
\end{equation}
and the corresponding stationary process $\{X_k\}$ is a $d$-variate AR$(m_0)$ process.

\begin{remark}
It should be noticed that the expression (\ref{eq:hinverse162}) with 
(\ref{eq:hinverse163}) is uniquely determined, 
up to a constant unitary factor, from $\{X_k\}$ satisfying (\ref{eq:farima529}) with (\ref{eq:C}) 
since so is $h$ in the factorization (\ref{eq:farima529}) with (\ref{eq:C}) (see Section \ref{sec:2}). 
Suppose that we start with a $d$-variate, causal and invertible ARMA process $\{X_k\}$ in the sense 
of \cite{BD}, that is, a $\C^d$-valued, centered, weakly stationary process described by 
the ARMA equation
\[
\Phi(B) X_n = \Psi(B) \xi_n,\qquad n\in\Z,
\]
where, for $r, s\in\N\cup\{0\}$ and $\Phi_i, \Psi_j\in\C^{d\times d}\ (i=1,\dots,r,\ j=1,\dots,s)$, 
\[
\Phi(z) = I_d - z\Phi_1 - \cdots - z^r\Phi_r 
\quad \mbox{and} \quad 
\Psi(z) = I_d - z\Psi_1 - \cdots - z^s\Psi_s
\]
are $\C^{d\times d}$-valued polynomials satisfying 
$\det \Phi(z)\neq 0$ and $\det \Psi(z)\neq 0$ on $\overline{\D}$, 
$B$ is the backward shift operator defined by $B X_m=X_{m-1}$, 
and $\{\xi_k : k\in\Z\}$ is a $d$-variate white noise, that is, a $d$-variate, centered process 
such that $E[\xi_n \xi_m^*]=\delta_{nm}V$ for some positive-definite $V\in\C^{d\times d}$. 
Notice that the pair $(\Phi(z),\Psi(z))$ is not uniquely determined from $\{X_k\}$; for example, 
we can replace $(\Phi(z),\Psi(z))$ by $((2-z)\Phi(z),(2-z)\Psi(z))$. 
However, if we put $h(z) = \Phi(z)^{-1} \Psi(z) V^{1/2}$, 
then $h$ is an outer function belonging to $H_2^{d\times d}(\T)$ and satisfies 
(\ref{eq:farima529}) for the spectral density $w$ of $\{X_k\}$. 
Therefore, $h$ is uniquely determined, 
up to a constant unitary factor, from $\{X_k\}$. 
In particular, the expression (\ref{eq:hinverse162}) with (\ref{eq:hinverse163}) 
for $h$ is also uniquely determined, 
up to a constant unitary factor, from $\{X_k\}$. 
From these observations and the results in \cite{I20} and this paper, 
we are led to the idea of parameterizing the ARMA processes by 
the expression (\ref{eq:hinverse162}) with (\ref{eq:hinverse163}) (see Remark 8 in \cite{I20}). 
This point will be discussed in future work.
\end{remark}

By Theorem 2 in \cite{I20}, 
$h_{\sharp}^{-1}$ has the same $m_0$ and the same poles 
with the same multiplicities as $h^{-1}$, 
that is, 
for $m_0$, $K$ and $(p_1, m_1), \dots, (p_K, m_K)$ in 
$(\ref{eq:hinverse162})$ with $(\ref{eq:hinverse163})$, 
$h_{\sharp}^{-1}$ has the form
\begin{equation}
h_{\sharp}(z)^{-1} = - \rho_{0,0}^{\sharp} - 
\sum_{{\mu}=1}^{K} \sum_{j=1}^{m_{\mu}} \frac{1}{(1-\barp_{\mu}z)^j}\rho_{{\mu}, j}^{\sharp} 
- \sum_{j=1}^{m_0} z^j \rho_{0,j}^{\sharp},
\label{eq:hsharpinv162}
\end{equation}
where
\[
\left\{
\begin{aligned}
&\rho_{{\mu}, j}^{\sharp}\in\C^{d\times d}, \quad 
\mu \in \{0,\dots,K\},\ j \in \{1,\dots,m_{\mu}\}, \qquad 
\rho_{0,0}^{\sharp} \in\C^{d\times d},\\
&\rho_{{\mu},m_{\mu}}^{\sharp}\neq 0, \quad \mu \in \{1,\dots,K\},\\
&\rho_{0,m_0}^{\sharp}\neq 0\quad \mbox{if}\ \ m_0\ge 1.
\end{aligned}
\right.
\]
Notice that if $d=1$, then we can take $h_{\sharp}=h$, hence $\rho_{0,0}=\rho_{0,0}^{\sharp}$ 
and $\rho_{{\mu}, j} = \rho_{{\mu}, j}^{\sharp}$ for $\mu\in\{1,\dots,K\}$ and $j\in\{1,\dots,m_{\mu}\}$.

Recall $\tilde{h}$ from (\ref{eq:outft628}). From (\ref{eq:hsharpinv162}), 
we have
\[
\tilde{h}(z)^{-1} = - \tilde{\rho}_{0,0} - 
\sum_{{\mu}=1}^{K} \sum_{j=1}^{m_{\mu}} \frac{1}{(1 - p_{\mu} z)^j} \tilde{\rho}_{{\mu}, j} 
- \sum_{j=1}^{m_0} z^j \tilde{\rho}_{0,j},
\]
where
\[
\tilde{\rho}_{0,0} := (\rho_{0,0}^{\sharp})^*, \qquad 
\tilde{\rho}_{\mu,j} := (\rho_{\mu, j}^{\sharp})^*, \quad 
\mu \in \{0,\dots,K\},\ j \in \{1,\dots,m_{\mu}\}.
\]
Recall the sequences $\{a_k\}$ and $\{\tilde{a}_k\}$ from (\ref{eq:AR111}) and (\ref{eq:AR222}), respectively. 
We have
\begin{align}
a_n         &= \sum_{{\mu}=1}^{K} \sum_{j=1}^{m_{\mu}} \binom{n+j-1}{j-1} \barp_{\mu}^n
\rho_{{\mu}, j}, \qquad n\ge m_0 + 1,
\label{eq:a363}\\
\tilde{a}_n &= \sum_{{\mu}=1}^{K} \sum_{j=1}^{m_{\mu}} \binom{n+j-1}{j-1} p_{\mu}^n
\tilde{\rho}_{{\mu}, j}, \qquad n\ge m_0 + 1
\label{eq:atilde361}
\end{align}
and
\begin{align}
a_n         &= \rho_{0,n} + \sum_{{\mu}=1}^{K} \sum_{j=1}^{m_{\mu}} \binom{n+j-1}{j-1} \barp_{\mu}^n
\rho_{{\mu}, j}, \qquad n \in \{0,\dots,m_0\},
\label{eq:a364}\\
\tilde{a}_n &= \tilde{\rho}_{0,n} + \sum_{{\mu}=1}^{K} \sum_{j=1}^{m_{\mu}} \binom{n+j-1}{j-1} p_{\mu}^n
\tilde{\rho}_{{\mu}, j}, \qquad  n \in \{0,\dots,m_0\},
\label{eq:atilde362}
\end{align}
where the convention $\binom{0}{0}=1$ is adopted; see Proposition 4 in \cite{I20}.

We first consider the case of $K = 0$ that corresponds to a $d$-variate AR$(m_0)$ process. 
As can be seen from the following theorem, 
in this case, we have simple closed-form formulas for $T_n(w)^{-1}$.

\begin{theorem}\label{thm:TforR777}
We assume (\ref{eq:farima529}), (\ref{eq:C}) and $K = 0$ 
for $K$ in (\ref{eq:hinverse162}). 
Thus we assume (\ref{eq:K000}). 
Then the following four assertions hold.

\begin{enumerate}

\item For $n \ge m_0 + 1$, $s \in \{1,\dots,n\}$ and $t \in \{1, \dots, n-m_0\}$, we have
\begin{equation}
\left(T_n(w)^{-1}\right)^{s,t} 
= \sum_{\lambda=1}^{s\wedge t} \tilde{a}_{s-\lambda}^* \tilde{a}_{t-\lambda}.
\label{eq:TforR777}
\end{equation}

\item For $n \ge m_0 + 1$, $s \in \{1, \dots, n-m_0\}$ and $t \in \{1,\dots,n\}$, we have
\begin{equation}
\left(T_n(w)^{-1}\right)^{s,t} 
= \sum_{\lambda=1}^{s\wedge t} \tilde{a}_{s-\lambda}^* \tilde{a}_{t-\lambda}.
\label{eq:TforR999}
\end{equation}

\item 
For $n \ge m_0 + 1$, $s \in \{1,\dots,n\}$ and $t \in \{m_0+1, \dots, n\}$, we have
\begin{equation}
\left(T_n(w)^{-1}\right)^{s,t} 
= \sum_{\lambda=s\vee t}^n a_{\lambda-s}^* a_{\lambda-t}.
\label{eq:TforR666}
\end{equation}

\item 
For $n \ge m_0 + 1$, $s \in \{m_0+1, \dots, n\}$ and $t \in \{1,\dots,n\}$, we have
\begin{equation}
\left(T_n(w)^{-1}\right)^{s,t} 
= \sum_{\lambda=s\vee t}^n a_{\lambda-s}^* a_{\lambda-t}.
\label{eq:TforR888}
\end{equation}
\end{enumerate}

\end{theorem}

\begin{proof}
For $w$ satisfying (\ref{eq:farima529}), (\ref{eq:C}) and $K = 0$, 
let $\{X_k\}$, $\{X^{\prime}_k\}$, $\{\varepsilon_k\}$ and $\{\tilde{\varepsilon}_k\}$ 
be as in Section \ref{sec:3}.

(i)\ By (\ref{eq:hsharpinv162}) with $K=0$, we have $\tilde{a}_0=\tilde{\rho}_0$, 
$\tilde{a}_k=\tilde{\rho}_{0,k}$ for $k\in\{1,\dots,m_0\}$ and $\tilde{a}_k=0$ for 
$k \ge m_0 + 1$. In particular, we have 
$\sum_{k=0}^{m_0}\tilde{a}_{k}X_{u+k} + \tilde{\varepsilon}_{-u} = 0$ for $u\in\Z$; see (2.15) in \cite{IKP2}. 
This implies $\tilde{\varepsilon}_{-u} \in V_{[1,n]}^X$, or $P_{[1,n]}^{\bot} \tilde{\varepsilon}_{-u} = 0$, 
for $u\in\{1,\dots,n-m_0\}$. 
Therefore, (\ref{eq:TforR777}) follows from Theorem \ref{thm:Tdual123} and (\ref{eq:dual888}). 

(iii)\ By (\ref{eq:K000}), we have $a_0=\rho_{0,0}$, 
$a_k=\rho_{0,k}$ for $k\in\{1,\dots,m_0\}$ and $a_k=0$ for $k \ge m_0 + 1$. In particular, 
$\sum_{k=0}^{m_0}a_{k}X_{u-k} + \varepsilon_u = 0$ for $u\in\Z$; see (2.15) in \cite{IKP2}. 
This implies $\varepsilon_u \in V_{[1,n]}^X$, or $P_{[1,n]}^{\bot} \varepsilon_u = 0$, 
for $u\in\{m_0+1, \dots,n\}$. 
Therefore, (\ref{eq:TforR666}) follows from Theorem \ref{thm:Tdual123} and (\ref{eq:dual777}).

(ii), (iv)\ By (\ref{eq:SA-123}), (ii) and (iv) follow from (i) and (iii), respectively.
\end{proof}

We turn to the case of $K \ge 1$. 
In what follows in this section, for $K$ in (\ref{eq:hinverse162}), we assume
\[
K\ge 1.
\]
For $m_1,\dots,m_K$ in (\ref{eq:hinverse162}), we define $M\in \N$ by
\begin{equation}
M:=\sum_{\mu=1}^K m_{\mu}.
\label{eq:sum-mmu123}
\end{equation}

For $\mu\in\{1,\dots,K\}$, $p_{\mu}$ in (\ref{eq:hinverse162}) and $i\in\N$, 
we define $p_{{\mu}, i}: \Z \to \C^{d\times d}$ by
\begin{equation}
p_{{\mu},i}(k):=
\binom{k}{i-1} p_{\mu}^{k - i+1} I_d, \qquad k\in\Z.
\label{eq:p362}
\end{equation}
Notice that
\[
p_{\mu,i}(0)=\binom{0}{i-1}p_{\mu}^{-i+1} I_d = \delta_{i, 1} I_d.
\]
For $n\in \Z$, we also define $\p_n \in \C^{dM\times d}$ by the 
following block representation:
\[
\begin{aligned}
\p_n
&:=(p_{1, 1}(n), \dots, p_{1, m_1}(n)\ \vert\ p_{2, 1}(n), \dots, p_{2, m_2}(n)\ \vert \\
&\qquad\qquad\qquad\qquad\qquad\qquad\qquad 
\cdots \vert\ 
p_{K, 1}(n), \dots, p_{K, m_{K}}(n))^{\top}.
\end{aligned}
\]
Notice that
\[
\p_0 = (I_d, 0, \dots, 0\ \vert\ I_d, 0, \dots, 0 \vert 
\cdots \vert\ 
I_d, 0, \dots, 0)^{\top} \in \C^{dM\times d}.
\]

We define $\Lambda\in\C^{dM\times dM}$ by
\[
\Lambda := \sum_{\ell=0}^{\infty} \p_{\ell} \p_{\ell}^*.
\]
For ${\mu}, {\nu}\in\{1,2,\dots,K\}$, we define $\Lambda^{{\mu},{\nu}}\in \C^{dm_{\mu}\times dm_{\nu}}$ by the 
block representation
\[
\Lambda^{{\mu},{\nu}} := \left(
\begin{matrix}
\lambda^{{\mu}, {\nu}}(1, 1) & \lambda^{{\mu}, {\nu}}(1, 2) & \cdots & \lambda^{{\mu}, {\nu}}(1, m_{\nu}) \cr
\lambda^{{\mu}, {\nu}}(2, 1) & \lambda^{{\mu}, {\nu}}(2, 2) & \cdots & \lambda^{{\mu}, {\nu}}(2, m_{\nu}) \cr
\vdots                       &  \vdots                      &        & \vdots                             \cr
\lambda^{\mu,\nu}(m_{\mu},1) & \lambda^{\mu,\nu}(m_{\mu},2) & \cdots & \lambda^{\mu,\nu}(m_{\mu},m_{\nu}) \cr
\end{matrix}
\right),
\]
where, for $i \in \{1,\dots,m_{\mu}\}$ and $j \in \{1,\dots,m_{\nu}\}$,
\[
\lambda^{{\mu},{\nu}}(i, j) 
:= \sum_{r=0}^{j-1} 
\binom{i-1}{r} \binom{i+j-r-2}{i-1} 
\frac{p_{\mu}^{j - r -1}\barp_{\nu}^{i-r-1}}{(1-p_{\mu}\barp_{\nu})^{i+j-r-1}} I_d 
\in \C^{d\times d}.
\]
Then, by Lemma 3 in \cite{I20}, the matrix $\Lambda$ has the following block representation:
\[
\Lambda=\left(
\begin{matrix}
\Lambda^{1, 1}       &  \Lambda^{1, 2}   &  \cdots  & \Lambda^{1, K}  \cr
\Lambda^{2, 1}       &  \Lambda^{2, 2}   &  \cdots  & \Lambda^{2, K}  \cr
\vdots               &  \vdots           &  \ddots  & \vdots          \cr
\Lambda^{K, 1}       &  \Lambda^{K, 2}   &  \cdots  & \Lambda^{K, K}  \cr
\end{matrix}
\right).
\]

We define, for $\mu \in \{1,\dots,K\}$ and $j \in \{1,\dots,m_{\mu}\}$,
\begin{equation}
\theta_{{\mu}, j} := - \lim_{z\to p_{\mu}} \frac{1}{(m_{\mu} - j)!} 
\frac{d^{m_{\mu} - j}}{dz^{m_{\mu} - j}} \left\{(z-p_{\mu})^{m_{\mu}} h_{\sharp}(z) h^{\dagger}(z)^{-1}\right\} 
\in \C^{d\times d},
\label{eq:theta456}
\end{equation}
where
\begin{equation}
h^{\dagger}(z) := h(1/\overline{z})^*.
\label{eq:hdagger456}
\end{equation}
We define 
$\Theta\in \C^{dM\times dM}$ by the block representation
\[
\Theta:=
\left(
\begin{matrix}
\Theta_1 & 0        & \cdots & 0\cr
0        & \Theta_2 & \cdots & 0\cr
\vdots   & \vdots   & \ddots & \vdots      \cr
0        & 0        & \cdots & \Theta_{K}
\end{matrix}
\right),
\]
where, for $\mu \in \{1,\dots,K\}$, $\Theta_{\mu} \in \C^{dm_{\mu}\times dm_{\mu}}$ is defined by
\[
\Theta_{\mu} := 
\left(
\begin{matrix}
\theta_{{\mu}, 1} & \theta_{{\mu}, 2}  &  \cdots & \theta_{{\mu}, m_{\mu}-1} & \theta_{{\mu}, m_{\mu}} \cr
\theta_{{\mu}, 2} & \theta_{{\mu}, 3}  &  \cdots & \theta_{{\mu}, m_{\mu}}   &                         \cr
\vdots            &  \vdots            &         &                           &                         \cr
\theta_{{\mu}, m_{\mu}-1}   &  \theta_{{\mu}, m_{\mu}} &          &          &                         \cr
\theta_{{\mu}, m_{\mu}}     &                          &          &          & \mbox{\huge 0}
\end{matrix}
\right)
\]
using $\theta_{\mu,j}$ in (\ref{eq:theta456}) with (\ref{eq:hdagger456}).

For $n\in \Z$, we define 
$\Pi_n\in \C^{dM\times dM}$ by the block representation
\[
\Pi_n:=
\left(
\begin{matrix}
\Pi_{1, n} & 0          & \cdots & 0\cr
0          & \Pi_{2, n} & \cdots & 0\cr
\vdots     & \vdots     & \ddots & \vdots      \cr
0          & 0          & \cdots & \Pi_{K, n}
\end{matrix}
\right),
\]
where, for $\mu \in \{1,\dots,K\}$ and $n\in\Z$, $\Pi_{\mu, n} \in \C^{dm_{\mu}\times dm_{\mu}}$ is defined by
\[
\Pi_{{\mu},n} := 
\left(
\begin{matrix}
p_{{\mu}, 1}(n) & p_{{\mu}, 2}(n) & p_{{\mu}, 3}(n) & \cdots & p_{{\mu}, m_{\mu}}(n)   \cr
                & p_{{\mu}, 1}(n) & p_{{\mu}, 2}(n) & \cdots & p_{{\mu}, m_{\mu}-1}(n) \cr
                &                 &     \ddots      & \ddots &    \vdots               \cr
                &                 &                 & \ddots & p_{{\mu}, 2}(n)         \cr
\mbox{\huge 0}  &                 &                 &        & p_{{\mu}, 1}(n)
\end{matrix}
\right)
\]
using $p_{{\mu}, i}(n)$ in (\ref{eq:p362}).

The next lemma slightly extends Lemma 17 in \cite{I20}.

\begin{lemma}\label{lem:beta162}
We assume (\ref{eq:farima529}), (\ref{eq:C}) and $K \ge 1$ for $K$ in (\ref{eq:hinverse162}). 
Then, 
for $n, k, \ell\in\Z$ such that $n+k+\ell\ge m_0$, we have
\[
\beta_{n+k+\ell+1}^* = \p_{\ell}^{\top} \Pi_n \Theta \p_k,
\]
hence
\[
\beta_{n+k+\ell+1} = \p_k^* (\Pi_n \Theta)^* \overline{\p}_{\ell}.
\]
\end{lemma}

The proof of Lemma \ref{lem:beta162} is almost the same as that of Lemma 17 in \cite{I20}, hence we omit it.

For $n\in \Z$, we define $G_n, \tilde{G}_n\in \C^{dM\times dM}$ by
\[
G_n :=  \Pi_n \Theta \Lambda, 
\qquad 
\tilde{G}_n := (\Pi_n \Theta)^* \Lambda^{\top}.
\]

\begin{lemma}\label{lem:bk123}
We assume (\ref{eq:farima529}), (\ref{eq:C}) and $K \ge 1$ 
for $K$ in (\ref{eq:hinverse162}). Then the following two 
assertions hold.
\begin{enumerate}
\item 
We assume $n\ge u\ge m_0+1$. Then, for $k\in\N$ and $\ell\in\mathbb{N}\cup\{0\}$, we have
\begin{align}
b_{n,u,\ell}^{2k-1} &= \p_{u-n-1}^* ( \tilde{G}_n G_n )^{k-1} (\Pi_n \Theta)^* \overline{\p}_{\ell},
\label{eq:bodd123}\\
b_{n,u,\ell}^{2k} &= \p_{u-n-1}^*  (\tilde{G}_n G_n )^{k-1} \tilde{G}_n \Pi_n \Theta \p_{\ell}.
\label{eq:beven123}
\end{align}
\item 
We assume $1\le u\le n-m_0$. Then, for $k\in\N$ and $\ell\in\mathbb{N}\cup\{0\}$, we have
\begin{align}
\tilde{b}_{n,u,\ell}^{2k-1} &= \p_{-u}^{\top} ( G_n \tilde{G}_n )^{k-1} \Pi_n \Theta \p_{\ell},
\label{eq:til-bodd123}\\
  \tilde{b}_{n,u,\ell}^{2k} &= \p_{-u}^{\top}  (G_n \tilde{G}_n )^{k-1} G_n (\Pi_n \Theta)^* \overline{\p}_{\ell}.
\label{eq:til-beven123}
\end{align}
\end{enumerate}
\end{lemma}

The proof of Lemma \ref{lem:bk123} will be given in the Appendix.

For $n\in\N$ and ${\mu}, {\nu}\in\{1,2,\dots,K\}$, 
we define $\Xi_n^{{\mu},{\nu}}\in \C^{dm_{\mu}\times dm_{\nu}}$ by 
the block representation
\[
\Xi_n^{{\mu},{\nu}} := \left(
\begin{matrix}
\xi_n^{{\mu}, {\nu}}(1, 1) & \xi_n^{{\mu}, {\nu}}(1, 2) & \cdots & \xi_n^{{\mu}, {\nu}}(1, m_{\nu}) \cr
\xi_n^{{\mu}, {\nu}}(2, 1) & \xi_n^{{\mu}, {\nu}}(2, 2) & \cdots & \xi_n^{{\mu}, {\nu}}(2, m_{\nu}) \cr
\vdots                       &  \vdots                      &        & \vdots                       \cr
\xi_n^{\mu,\nu}(m_{\mu},1) & \xi_n^{\mu,\nu}(m_{\mu},2) & \cdots & \xi_n^{\mu,\nu}(m_{\mu},m_{\nu}) \cr
\end{matrix}
\right),
\]
where, for $n\in\N$, $i \in \{1,\dots,m_{\mu}\}$ and $j \in \{1,\dots,m_{\nu}\}$, 
$\xi_n^{{\mu},{\nu}}(i, j) \in \C^{d\times d}$ is defined by
\[
\xi_n^{{\mu},{\nu}}(i, j)
:= \sum_{r=0}^{j-1} 
\binom{n+i+j-2}{r} \binom{i+j-r-2}{i-1} 
\frac{p_{\mu}^{j - r -1}\barp_{\nu}^{n+i+j-r-2}}{(1-p_{\mu}\barp_{\nu})^{i+j-r-1}} I_d.
\]
For $n\in\N$, we define $\Xi_n\in\C^{dM\times dM}$ by
\[
\Xi_n:=\left(
\begin{matrix}
\Xi_n^{1, 1}       &  \Xi_n^{1, 2}   &  \cdots  & \Xi_n^{1, K}  \cr
\Xi_n^{2, 1}       &  \Xi_n^{2, 2}   &  \cdots  & \Xi_n^{2, K}  \cr
\vdots             &  \vdots         &  \ddots  & \vdots        \cr
\Xi_n^{K, 1}       &  \Xi_n^{K, 2}   &  \cdots  & \Xi_n^{K, K}  \cr
\end{matrix}
\right).
\]
We also define $\rho \in \C^{dM\times d}$ and $\tilde{\rho} \in \C^{dM\times d}$ 
by the block representations
\[
\rho
:=(\rho_{1, 1}^{\top}, \dots, \rho_{1, m_1}^{\top} \ \vert \ 
\rho_{2, 1}^{\top}, \dots, \rho_{2, m_2}^{\top} \ \vert \ 
\cdots \ \vert \ 
\rho_{K, 1}^{\top}, \dots, \rho_{K, m_{K}}^{\top})^{\top}
\]
and
\[
\begin{aligned}
\tilde{\rho}
&:=(\tilde{\rho}_{1, 1}^{\top}, \dots, \tilde{\rho}_{1, m_1}^{\top} 
\ \vert \ \tilde{\rho}_{2, 1}^{\top}, \dots, \tilde{\rho}_{2, m_2}^{\top} \ \vert \ 
\cdots \ \vert \ 
\tilde{\rho}_{K, 1}^{\top}, \dots, \tilde{\rho}_{K, m_{K}}^{\top})^{\top}\\
&=\left(\overline{\rho_{1, 1}^{\sharp}}, \dots, \overline{\rho_{1, m_1}^{\sharp}} 
\ \vert \ \overline{\rho_{2, 1}^{\sharp}}, \dots, \overline{\rho_{2, m_2}^{\sharp}} \ \vert \ 
\cdots \ \vert \ 
\overline{\rho_{K, 1}^{\sharp}}, \dots, \overline{\rho_{K, m_{K}}^{\sharp}}\right)^{\top},
\end{aligned}
\]
respectively. 
For $n\in\N$, we define $v_n, \tilde{v}_n \in\C^{dM\times d}$ by
\[
v_n := \sum_{\ell=0}^{\infty} \p_{\ell} a_{n+\ell},
\qquad
\tilde{v}_n := \sum_{\ell=0}^{\infty} \overline{\p}_{\ell} \tilde{a}_{n+\ell}.
\]
Then, by Lemma 5 in \cite{I20}, we have
\[
v_n         = \Xi_n \rho, 
\qquad 
\tilde{v}_n =  \overline{\Xi}_n \tilde{\rho}, 
\qquad 
n\ge m_0 + 1.
\]
Moreover, if $m_0\ge 1$, then we have
\[
v_n         =  \Xi_n \rho + \sum_{\ell=0}^{m_0-n} \p_{\ell} \rho_{0,n+\ell}, 
\quad 
\tilde{v}_n =  \overline{\Xi}_n \tilde{\rho} + \sum_{\ell=0}^{m_0-n} \overline{\p}_{\ell} \tilde{\rho}_{0,n+\ell},
 \quad  n \in \{1,\dots,m_0\}.
\]

For $n\in\Z$, we define $w_n, \tilde{w}_n \in\C^{dM\times d}$ by
\[
w_n := \sum_{\ell=0}^{\infty} \p_{\ell - n} a_{\ell},
\qquad 
\tilde{w}_n := \sum_{\ell=0}^{\infty} \overline{\p}_{\ell - n} \tilde{a}_{\ell}.
\]
To give closed-form expressions for $w_n$ and $\tilde{w}_n$, we introduce some matrices. 
For $n\in\Z$ and ${\mu}, {\nu}\in\{1,2,\dots,K\}$, 
we define $\Phi_n^{{\mu},{\nu}}\in \C^{dm_{\mu}\times dm_{\nu}}$ by 
the block representation
\[
\Phi_n^{{\mu},{\nu}} := \left(
\begin{matrix}
\varphi_n^{{\mu}, {\nu}}(1, 1) & \varphi_n^{{\mu}, {\nu}}(1, 2) & \cdots & \varphi_n^{{\mu}, {\nu}}(1, m_{\nu}) \cr
\varphi_n^{{\mu}, {\nu}}(2, 1) & \varphi_n^{{\mu}, {\nu}}(2, 2) & \cdots & \varphi_n^{{\mu}, {\nu}}(2, m_{\nu}) \cr
\vdots                         &  \vdots                        &       & \vdots                                \cr
\varphi_n^{\mu,\nu}(m_{\mu},1) & \varphi_n^{\mu,\nu}(m_{\mu},2) & \cdots & \varphi_n^{\mu,\nu}(m_{\mu},m_{\nu}) \cr
\end{matrix}
\right),
\]
where, for $n\in\Z$, $i=1,\dots,m_{\mu}$ and $j = 1,\dots,m_{\nu}$, 
$\varphi_n^{{\mu},{\nu}}(i, j) \in \C^{d\times d}$ is defined by
\[
\varphi_n^{{\mu},{\nu}}(i, j)
:= \sum_{q=0}^{i-1} \sum_{r=0}^{j-1}  
\binom{j-1}{r} \binom{r+q}{q} \binom{r - n}{i-q-1} 
  \frac{p_{\mu}^{r+q+1-i - n} \barp_{\nu}^{r+q}}{(1-p_{\mu}\barp_{\nu})^{r+q+1}} I_d.
\]
For $n\in\Z$, we define $\Phi_n\in\C^{dM\times dM}$ by
\[
\Phi_n:=\left(
\begin{matrix}
\Phi_n^{1, 1}  &  \Phi_n^{1, 2}  &  \cdots  & \Phi_n^{1, K}  \cr
\Phi_n^{2, 1}  &  \Phi_n^{2, 2}  &  \cdots  & \Phi_n^{2, K}  \cr
\vdots         &  \vdots         &  \ddots  & \vdots         \cr
\Phi_n^{K, 1}  &  \Phi_n^{K, 2}  &  \cdots  & \Phi_n^{K, K}  \cr
\end{matrix}
\right).
\]

Here are closed-form expressions for $w_n$ and $\tilde{w}_n$.

\begin{lemma}\label{lem:w516}
We have
\begin{align*}
w_n         &=  \Phi_n \rho +  \sum_{\ell=0}^{m_0} \p_{\ell - n} \rho_{0,\ell}, \qquad n \in \Z,
\\
\tilde{w}_n &=  \overline{\Phi}_n \tilde{\rho} 
+ \sum_{\ell=0}^{m_0} \overline{\p}_{\ell - n} \tilde{\rho}_{0,\ell}, \qquad n \in \Z.
\end{align*}
\end{lemma}

The proof of Lemma \ref{lem:w516} will be given in the Appendix.

Recall $M$ from (\ref{eq:sum-mmu123}). 
For $n \in \N$ and $s\in \{1,\dots,n\}$, 
we define 
\begin{align*}
        \ell_{n,s} &:= \{w_{n+1-s} - v_{n+1-s}\}^* (I_{dM} - \tilde{G}_n G_n )^{-1} \in \C^{d \times dM},
\\
\tilde{\ell}_{n,s} &:= \{\tilde{w}_{s} - \tilde{v}_s\}^* (I_{dM} - G_n \tilde{G}_n )^{-1} \in \C^{d \times dM},
\\\
r_{n,s} &:= (\Pi_n \Theta)^* \tilde{v}_s + \tilde{G}_n \Pi_n \Theta v_{n+1-s} \in \C^{dM \times d}
\end{align*}
and
\[
\tilde{r}_{n,s} := \Pi_n \Theta v_{n+1-s} + G_n (\Pi_n \Theta)^* \tilde{v}_s \in \C^{dM \times d}.
\]

Here are closed-form formulas for $(T_n(w))^{-1}$ with $w$ 
satisfying (\ref{eq:C}) and $K \ge 1$.

\begin{theorem}\label{thm:TforR123}
We assume (\ref{eq:farima529}), (\ref{eq:C}) and $K \ge 1$ for $K$ in (\ref{eq:hinverse162}). 
Then the following four assertions hold.
\begin{enumerate}

\item For $n \ge m_0 + 1$, $s \in \{1,\dots,n\}$ and $t \in \{1, \dots, n-m_0\}$, we have
\[
\left(T_n(w)^{-1}\right)^{s,t} 
= \tilde{r}_{n,s}^* \tilde{\ell}_{n,t}^*  
+ \sum_{\lambda=1}^{s\wedge t} \tilde{a}_{s-\lambda}^* \tilde{a}_{t-\lambda}.
\]

\item For $n \ge m_0 + 1$, $s \in \{1, \dots, n-m_0\}$ and $t \in \{1,\dots,n\}$, we have
\[
\left(T_n(w)^{-1}\right)^{s,t} 
= \tilde{\ell}_{n,s} \tilde{r}_{n,t} 
+ \sum_{\lambda=1}^{s\wedge t} \tilde{a}_{s-\lambda}^* \tilde{a}_{t-\lambda}.
\]

\item 
For $n \ge m_0 + 1$, $s \in \{1,\dots,n\}$ and $t \in \{m_0+1, \dots, n\}$, we have
\[
\left(T_n(w)^{-1}\right)^{s,t} 
= r_{n,s}^* \ell_{n,t}^*   
+ \sum_{\lambda=s\vee t}^n a_{\lambda-s}^* a_{\lambda-t}.
\]

\item 
For $n \ge m_0 + 1$, $s \in \{m_0+1, \dots, n\}$ and $t \in \{1,\dots,n\}$, we have
\[
\left(T_n(w)^{-1}\right)^{s,t} 
= \ell_{n,s} r_{n,t} 
+ \sum_{\lambda=s\vee t}^n a_{\lambda-s}^* a_{\lambda-t}.
\]

\end{enumerate}
\end{theorem}

\begin{proof}
(i)\quad We assume $n \ge m_0 + 1$, $s \in \{1,\dots,n\}$ and $t \in \{1, \dots, n-m_0\}$. 
Then, by Lemma \ref{lem:bk123} (ii) above and Lemma 19 in \cite{I20}, we have
\[
\begin{aligned}
 &\sum_{u=1}^t \sum_{k=1}^{\infty} 
     \left\{ 
          \sum_{\lambda=0}^{\infty} \tilde{b}_{n,u,\lambda}^{2k-1} a_{n+1-s+\lambda} 
     \right\}^* \tilde{a}_{t-u}\\
&\qquad =\sum_{u=1}^t \sum_{k=1}^{\infty} 
     \left\{ 
          \sum_{\lambda=0}^{\infty} \p_{-u}^{\top} ( G_n \tilde{G}_n )^{k-1} \Pi_n \Theta 
                                                      \p_{\lambda} a_{n+1-s+\lambda} 
     \right\}^* \tilde{a}_{t-u}\\
&\qquad =\sum_{u=1}^t \sum_{k=1}^{\infty} 
     \left\{ 
          \p_{-u}^{\top} ( G_n \tilde{G}_n )^{k-1} \Pi_n \Theta v_{n+1-s} 
     \right\}^* \tilde{a}_{t-u}\\
&\qquad =\sum_{u=1}^t 
     \left\{ 
          \p_{-u}^{\top} ( I_{dM} - G_n \tilde{G}_n )^{-1} \Pi_n \Theta v_{n+1-s} 
     \right\}^* \tilde{a}_{t-u}\\
&\qquad =v_{n+1-s}^* (\Pi_n \Theta)^* (I_{dM} - \tilde{G}_n^* G_n^*)^{-1}
\sum_{u=1}^t  \overline{\p}_{-u} \tilde{a}_{t-u}.
\end{aligned}
\]
Similarly, by Lemma \ref{lem:bk123} (ii) above and Lemma 19 in \cite{I20}, 
\[
\sum_{u=1}^t \sum_{k=1}^{\infty} 
     \left\{ 
          \sum_{\lambda=0}^{\infty} \tilde{b}_{n,u,\lambda}^{2k} \tilde{a}_{s+\lambda} 
     \right\}^* \tilde{a}_{t-u}
= \tilde{v}_s^* \Pi_n \Theta G_n^* (I_{dM} - \tilde{G}_n^* G_n^*)^{-1}
\sum_{u=1}^t  \overline{\p}_{-u} \tilde{a}_{t-u}.
\]
However, 
$\sum_{u=1}^t  \overline{\p}_{-u} \tilde{a}_{t-u}
=\sum_{\lambda=0}^{\infty}  \overline{\p}_{\lambda-t} \tilde{a}_{\lambda}
-\sum_{\lambda=0}^{\infty}  \overline{\p}_{\lambda} \tilde{a}_{t+\lambda}
=\tilde{w}_{t} - \tilde{v}_t$. 
Therefore, the assertion (i) follows from Theorem \ref{thm:TforXwithM123} (i).

(iii)\quad We assume $n \ge m_0 + 1$, $s \in \{1,\dots,n\}$ and $t \in \{m_0+1, \dots, n\}$. 
Then, by Lemma \ref{lem:bk123} (i) above and Lemma 19 in \cite{I20}, we have
\[
\begin{aligned}
 &\sum_{u=t}^n \sum_{k=1}^{\infty} 
     \left\{ 
          \sum_{\lambda=0}^{\infty} b_{n,u,\lambda}^{2k-1} \tilde{a}_{s+\lambda} 
     \right\}^* a_{u-t}\\
&\qquad =\sum_{u=t}^n \sum_{k=1}^{\infty} 
     \left\{ 
          \sum_{\lambda=0}^{\infty} \p_{u-n-1}^* ( \tilde{G}_n G_n )^{k-1} (\Pi_n \Theta)^* \overline{\p}_{\lambda} 
\tilde{a}_{s+\lambda} 
     \right\}^* a_{u-t}\\
&\qquad =\sum_{u=t}^n \sum_{k=1}^{\infty} 
     \left\{ 
          \p_{u-n-1}^* ( \tilde{G}_n G_n )^{k-1} (\Pi_n \Theta)^* \tilde{v}_{s} 
     \right\}^* a_{u-t}\\
&\qquad =\sum_{u=t}^n 
     \left\{ 
          \p_{u-n-1}^* ( I_{dM} - \tilde{G}_n G_n )^{-1} (\Pi_n \Theta)^* \tilde{v}_{s} 
     \right\}^* a_{u-t}\\
&\qquad = \tilde{v}_{s}^* \Pi_n \Theta (I_{dM} - G_n^* \tilde{G}_n^* )^{-1}\sum_{u=t}^n  \p_{u-n-1} a_{u-t}.
\end{aligned}
\]
Similarly, by Lemma \ref{lem:bk123} (i) above and Lemma 19 in \cite{I20}, we have
\[
\begin{aligned}
&\sum_{u=t}^n \sum_{k=1}^{\infty} 
     \left\{ 
          \sum_{\lambda=0}^{\infty} b_{n,u,\lambda}^{2k} a_{n+1-s+\lambda} 
     \right\}^* a_{u-t}\\
&\qquad 
= v_{n+1-s}^* (\Pi_n \Theta)^* \tilde{G}_n^* (I_{dM} - G_n^* \tilde{G}_n^* )^{-1} 
\sum_{u=t}^n  \p_{u-n-1} a_{u-t}.
\end{aligned}
\]
However, $\sum_{u=t}^n  \p_{u-n-1} a_{u-t} = w_{n+1-t} - v_{n+1-t}$. 
Therefore, the assertion (ii) follows from Theorem \ref{thm:TforXwithM123} (ii).

(ii), (iv)\quad By (\ref{eq:SA-123}), (ii) and (iv) follow from (i) and (iii), respectively.
\end{proof}

\begin{example}\label{exmp:simple333}
Suppose that $K\ge 1$, $m_{\mu}=1$ for $\mu \in \{1,\dots,K\}$ and $m_0=0$. Then, 
\[
h(z)^{-1} = - \rho_{0,0} - \sum_{{\mu}=1}^{K} \frac{1}{1-\barp_{\mu}z}\rho_{{\mu}, 1}, \qquad 
h_{\sharp}(z)^{-1} = - \rho^{\sharp}_{0,0} - 
\sum_{{\mu}=1}^{K} \frac{1}{1 - \barp_{\mu} z} \rho^{\sharp}_{{\mu}, 1}.
\]
We have
\begin{align*}
&\p_n^{\top} = (p_1^n I_d, \dots, p_K^n I_d)\in \C^{d \times dK}, \quad n \in \Z,\\
&\rho^{\top}
=(\rho_{1, 1}^{\top}, \rho_{2, 1}^{\top}, \dots, \rho_{K, 1}^{\top}) \in \C^{dK\times d}, 
\quad 
\tilde{\rho}^{\top}
=\left(\overline{\rho^{\sharp}_{1, 1}}, \overline{\rho^{\sharp}_{2, 1}}, 
\dots, \overline{\rho^{\sharp}_{K, 1}}\right)
\in \C^{dK\times d}.
\end{align*}
We also have
\begin{align*}
\Theta &= \left(
\begin{matrix}
p_1 h_{\sharp}(p_1) \rho_{1, 1}^* &    0                               &  \cdots &    0                              \cr
0                                 &  p_2 h_{\sharp}(p_2) \rho_{2, 1}^* &  \cdots &    0                              \cr
\vdots                            &  \vdots                            &  \ddots &   \vdots                          \cr
0                                 &    0                               &  \cdots & p_K h_{\sharp}(p_K) \rho_{K, 1}^*
\end{matrix}
\right)
\in \C^{dK\times dK},\\
\Lambda &= \left(
\begin{matrix}
\frac{1}{1-p_{1}\barp_{1}} I_d & \frac{1}{1-p_{1}\barp_{2}} I_d 
&  \cdots  & \frac{1}{1-p_{1}\barp_{K}} I_d  \cr
\frac{1}{1-p_{2}\barp_{1}} I_d  &  \frac{1}{1-p_{2}\barp_{2}} I_d  
&  \cdots  & \frac{1}{1-p_{2}\barp_{K}} I_d  \cr
\vdots               &  \vdots           &  \ddots       & \vdots          \cr
\frac{1}{1-p_{K}\barp_{1}} I_d & \frac{1}{1-p_{K}\barp_{2}} I_d 
&  \cdots  & \frac{1}{1-p_{K}\barp_{K}} I_d  \cr
\end{matrix}
\right)
\in \C^{dK\times dK},\\
\Pi_n &= \left(
\begin{matrix}
p_1^n I_d  &    0        &  \cdots &    0             \cr
0          & p_2^n I_d   &  \cdots &    0             \cr
\vdots     &  \vdots     &  \ddots &   \vdots         \cr
0          &    0        &  \cdots &   p_K^n I_d  
\end{matrix}
\right)
\in \C^{dK\times dK},\qquad n \in \Z,\\
\Xi_n &= \left(
\begin{matrix}
\frac{\barp_{1}^{n}}{1-p_{1}\barp_{1}} I_d & \frac{\barp_{2}^{n}}{1-p_{1}\barp_{2}} I_d 
&  \cdots  & \frac{\barp_{K}^{n}}{1-p_{1}\barp_{K}} I_d  \cr
\frac{\barp_{1}^{n}}{1-p_{2}\barp_{1}} I_d  &  \frac{\barp_{2}^{n}}{1-p_{2}\barp_{2}} I_d  
&  \cdots  & \frac{\barp_{K}^{n}}{1-p_{2}\barp_{K}} I_d  \cr
\vdots               &  \vdots           &  \ddots       & \vdots          \cr
\frac{\barp_{1}^{n}}{1-p_{K}\barp_{1}} I_d & \frac{\barp_{2}^{n}}{1-p_{K}\barp_{2}} I_d 
&  \cdots  & \frac{\barp_{K}^{n}}{1-p_{K}\barp_{K}} I_d  \cr
\end{matrix}
\right)
\in \C^{dK\times dK}, \qquad n \in \N,\\
\Phi_n &= \left(
\begin{matrix}
\frac{p_{1}^{-n}}{1-p_{1}\barp_{1}} I_d & \frac{p_{1}^{-n}}{1-p_{1}\barp_{2}} I_d 
&  \cdots  & \frac{p_{1}^{-n}}{1-p_{1}\barp_{K}} I_d  \cr
\frac{p_{2}^{-n}}{1-p_{2}\barp_{1}} I_d  &  \frac{p_{2}^{-n}}{1-p_{2}\barp_{2}} I_d  
&  \cdots  & \frac{p_{2}^{-n}}{1-p_{2}\barp_{K}} I_d  \cr
\vdots               &  \vdots           &  \ddots       & \vdots          \cr
\frac{p_{K}^{-n}}{1-p_{K}\barp_{1}} I_d & \frac{p_{K}^{-n}}{1-p_{K}\barp_{2}} I_d 
&  \cdots  & \frac{p_{K}^{-n}}{1-p_{K}\barp_{K}} I_d  \cr
\end{matrix}
\right)
\in \C^{dK\times dK}, \qquad n \in \Z,\\
G_n &=  \Pi_n \Theta \Lambda  \in \C^{dK\times dK}, 
\quad 
\tilde{G}_n = (\Pi_n \Theta)^* \Lambda^{\top} \in \C^{dK\times dK},
\quad 
n \in \Z,\\
v_n &= \Xi_n\rho \in \C^{dK\times d}, \quad \tilde{v}_n =  \overline{\Xi}_n \tilde{\rho} \in \C^{dK\times d}, 
\quad 
n\in\N,\\
w_n &= \Phi_n\rho + \p_{-n}\rho_{0,0}\in \C^{dK\times d}, 
\quad \tilde{w}_n =  \overline{\Phi}_n \tilde{\rho} + \overline{\p}_{-n} \tilde{\rho}_{0,0} \in \C^{dK\times d}, 
\quad 
n\in\Z.
\end{align*}
\end{example}

\begin{example}\label{exmp:simple444}
In Example \ref{exmp:simple333}, we further assume $d=K=1$. Then, we can write 
$h(z) = h_{\sharp}(z) = -(1 - \barp z)/\rho$, where $\rho\in \C\setminus\{0\}$ 
and $p \in \D\setminus\{0\}$. It follows that
\begin{align*}
&c_0=-1/\rho, \qquad c_1 = \barp/\rho, \qquad c_k = 0\quad (k\ge 2),\\
&a_k = \rho (\barp)^k, \qquad \tilde{a}_k = \overline{a}_k, \qquad k \in \N\cup\{0\}.
\end{align*}
Since $\gamma(k) = \sum_{\ell=0}^{\infty} c_{k+\ell}\overline{c}_{\ell}$ 
and $\gamma(-k)=\overline{\gamma(k)}$ for $k\in\N\cup\{0\}$, we have
\[
T_2(w) =
\frac{1}{\vert \rho\vert^2}
\left(
\begin{matrix}
1 + \vert p \vert^2 & -p                  \cr
-\barp       & 1 + \vert p \vert^2
\end{matrix}
\right),
\]
hence
\[
T_2(w)^{-1} =
\frac{\vert \rho\vert^2}{1 + \vert p \vert^2 + \vert p \vert^4}
\left(
\begin{matrix}
1 + \vert p \vert^2 & p                   \cr
\barp        & 1 + \vert p \vert^2 
\end{matrix}
\right).
\]
We also have
\[
\tilde{A}_2 = \overline{\rho}
\left(
\begin{matrix}
1 & p \cr
0 & 1
\end{matrix}
\right)
\quad 
\mbox{and}
\quad
A_2 = \rho
\left(
\begin{matrix}
1            & 0 \cr
\barp & 1
\end{matrix}
\right)
\]
for $\tilde{A}_2$ and $A_2$ in (\ref{eq:tildeAnU-234}) and (\ref{eq:AnL-234}), respectively. 
By simple calculations, we have
\begin{align*}
(\ell_{2,1}, \ell_{2,2}) 
&= 
\frac{\overline{\rho}}{(\barp)^2(1 - \vert p\vert^6)}
(1 + \vert p\vert^2, \barp), \qquad 
\left(\tilde{\ell}_{2,1}, \tilde{\ell}_{2,2}\right) 
= 
\left(\overline{\ell_{2,2}}, \overline{\ell_{2,1}} \right),\\
(r_{2,1}, r_{2,2}) 
&= 
- \rho \barp \vert p\vert^2 (1 - \vert p\vert^2) 
(\barp(1+\vert p\vert^2), \vert p\vert^2), \qquad 
\left(\tilde{r}_{2,1}, \tilde{r}_{2,2}\right) 
= 
\left(\overline{r_{2,2}}, \overline{r_{2,1}} \right)
\end{align*}
hence
\[
T_2(w)^{-1}
=
\tilde{A}_2^* \tilde{A}_2
+
\left(
\begin{matrix}
\tilde{\ell}_{2,1} \cr
\tilde{\ell}_{2,2} 
\end{matrix}
\right)
\left(
\tilde{r}_{2,1}, \tilde{r}_{2,2} 
\right)
=
A_2^* A_2
+
\left(
\begin{matrix}
\ell_{2,1} \cr
\ell_{2,2} 
\end{matrix}
\right)
\left(
r_{2,1}, r_{2,2} 
\right)
\]
which agrees with equalities in Theorem \ref{thm:TforR123}.
\end{example}

\section{Linear-time algorithm}\label{sec:6}

As in Section \ref{sec:5}, we assume (\ref{eq:farima529}) and (\ref{eq:C}). 
Let $K$ be as in (\ref{eq:hinverse162}) with (\ref{eq:hinverse163}). 
In this section, we explain how Theorems \ref{thm:TforR777} and \ref{thm:TforR123} 
above provide us with a linear-time algorithm to compute the solution $Z$ 
to the block Toeplitz system (\ref{eq:TS234}).

For
\begin{equation}
Y = (y_1^{\top},\dots,y_n^{\top})^{\top}\in\C^{dn\times d} \quad \mbox{with} \quad 
y_s \in \C^{d \times d}, \quad s \in \{1,\dots,n\},
\label{eq:Y-999}
\end{equation}
let
\[
Z = (z_1^{\top},\dots,z_n^{\top})^{\top}\in\C^{dn\times d} \quad \mbox{with} \quad 
z_s \in \C^{d \times d}, \quad s \in \{1,\dots,n\},
\]
be the solution to (\ref{eq:TS234}), that is, $Z = T_n(w)^{-1}Y$. 
For $m_0$ in (\ref{eq:hinverse162}), let $n \ge 2m_0 + 1$ so that $n-m_0 \ge m_0+1$ holds.

Recall $\tilde{A}_n$ and $A_n$ from (\ref{eq:tildeAnU-234}) and (\ref{eq:AnL-234}), respectively. 
If $K = 0$, then it follows from Lemma \ref{lem:decomp234} and Theorem \ref{thm:TforR777} (ii), (iv) that
\begin{align*}
z_s &= \tilde{\alpha}_{n,s}, \qquad s \in \{1,\dots,n-m_0\},
\\
z_s &= \alpha_{n,s}, \qquad s \in \{m_0+1,\dots,n\},
\end{align*}
where 
\begin{align*}
(\tilde{\alpha}_{n,1}^{\top}, \dots, \tilde{\alpha}_{n,n}^{\top})^{\top} 
            &:= \tilde{A}_n^* \tilde{A}_n Y
\quad \mbox{with} \quad \tilde{\alpha}_{n,s} \in \C^{d \times d}, \quad s \in \{1,\dots,n\},
\\
(\alpha_{n,1}^{\top}, \dots, \alpha_{n,n}^{\top})^{\top} 
            &:= A_n^* A_n Y
\quad \mbox{with} \quad \alpha_{n,s} \in \C^{d \times d}, \quad s \in \{1,\dots,n\}.
\end{align*}
On the other hand, if $K \ge 1$, then we see from Lemma \ref{lem:decomp234} and Theorem \ref{thm:TforR123} (ii), (iv) that
\begin{align*}
z_s &= \tilde{\ell}_{n,s} \tilde{R}_n 
             + \tilde{\alpha}_{n,s}, \qquad s \in \{1,\dots,n-m_0\},
\\
z_s &= \ell_{n,s} R_n 
             + \alpha_{n,s}, \qquad s \in \{m_0+1,\dots,n\},
\end{align*}
where 
\[
\tilde{R}_n := \sum_{t=1}^n \tilde{r}_{n,t} y_t \in \C^{d \times d}, 
\qquad
R_n := \sum_{t=1}^n r_{n,t} y_t \in \C^{d \times d}.
\]
Therefore, algorithms to compute $\tilde{A}_n^* \tilde{A}_n Y$ and $A_n^* A_n Y$ 
in $O(n)$ operations imply that of $Z$. 
We present the former ones below.

For $n\in\N\cup\{0\}$, $\mu\in\{1,\dots,K\}$ and $j\in\{1,\dots,m_{\mu}\}$, we define 
$q_{\mu,j}(n) \in \C^{d\times d}$ by $q_{\mu,j}(n):=p_{\mu,j}(n+j-1)$, that is,
\begin{equation}
q_{\mu,j}(n) = \binom{n+j-1}{j-1} p_{\mu}^nI_d.
\label{eq:a-224}
\end{equation}
For $n\in\N$, $\mu\in\{1,\dots,K\}$ and $j\in\{1,\dots,m_{\mu}\}$, we define 
the upper trianglular block Toeplitz matrix
$Q_{\mu,j,n} \in \C^{dn\times dn}$ by
\[
Q_{\mu,j,n} := 
\left(
\begin{matrix}
q_{\mu,j}(0)    & q_{\mu,j}(1) &  q_{\mu,j}(2) & \cdots & q_{\mu,j}(n-1) \cr
                & q_{\mu,j}(0) &  q_{\mu,j}(1) & \cdots & q_{\mu,j}(n-2) \cr
                &              &  \ddots       & \ddots &    \vdots      \cr
                &              &               & \ddots & q_{\mu,j}(1)   \cr
\mbox{\huge 0}  &              &               &        & q_{\mu,j}(0)
\end{matrix}
\right).
\]
Notice that
\[
Q^*_{\mu,j,n} := 
\left(
\begin{matrix}
q^*_{\mu,j}(0)   &                   &          &                 &  \mbox{\huge 0} \cr
q^*_{\mu,j}(1)   &  q^*_{\mu,j}(0)   &          &                 &                 \cr
q^*_{\mu,j}(2)   &  q^*_{\mu,j}(1)   &  \ddots  &                 &                 \cr
  \vdots         &    \vdots         &  \ddots  &  \ddots         &                 \cr
q^*_{\mu,j}(n-1) &  q^*_{\mu,j}(n-2) &  \cdots  &  q^*_{\mu,j}(1) &  q^*_{\mu,j}(0)
\end{matrix}
\right)
\]
with $q^*_{\mu,j}(n) = \binom{n+j-1}{j-1} \barp_{\mu}^nI_d$. 
For $n\in\N$, $\mu\in\{1,\dots,K\}$ and $j\in\{1,\dots,m_{\mu}\}$, we define 
the block diagonal matrices 
$\tilde{D}_{\mu,j,n} \in \C^{dn\times dn}$ 
and 
$D_{\mu,j,n} \in \C^{dn\times dn}$ by
\[
\tilde{D}_{\mu,j,n} := 
\left(
\begin{matrix}
\tilde{\rho}_{{\mu},j} & 0        & \cdots & 0                       \cr
0        & \tilde{\rho}_{{\mu},j} & \cdots & 0                       \cr
\vdots   & \vdots                 & \ddots & \vdots                  \cr
0        & 0                      & \cdots & \tilde{\rho}_{{\mu},j}
\end{matrix}
\right)
\quad \mbox{and} \quad 
D_{\mu,j,n} := 
\left(
\begin{matrix}
\rho_{{\mu},j} & 0              & \cdots & 0              \cr
0              & \rho_{{\mu},j} & \cdots & 0              \cr
\vdots         & \vdots         & \ddots & \vdots         \cr
0              & 0              & \cdots & \rho_{{\mu},j}
\end{matrix}
\right),
\]
respectively. 
Moreover, 
for $n\ge m_0+1$, we define the upper and lower triangular block Toeplitz matrices 
$\tilde{\Delta}_{n} \in \C^{dn\times dn}$ 
and 
$\Delta_{n} \in \C^{dn\times dn}$ by
\[
\tilde{\Delta}_{n} := 
\left(
\begin{matrix}
\tilde{\rho}_{0,0} & \tilde{\rho}_{0,1} & \cdots             & \tilde{\rho}_{0,m_0} &                      & \mbox{\huge 0}       \cr
                   & \tilde{\rho}_{0,0} & \tilde{\rho}_{0,1} &                      & \ddots               &                      \cr
                   &                    & \ddots             & \ddots               &                      & \tilde{\rho}_{0,m_0} \cr
                   &                    &                    & \ddots               & \ddots               & \vdots               \cr
                   &                    &                    &                      & \tilde{\rho}_{0,0}   & \tilde{\rho}_{0,1}   \cr
\mbox{\huge 0}     &                    &                    &                      &                      & \tilde{\rho}_{0,0}  
\end{matrix}
\right)
\]
and
\[
\Delta_{n} := 
\left(
\begin{matrix}
\rho_{0,0}     &            &               &        &            & \mbox{\huge 0}  \cr
\rho_{0,1}     & \rho_{0,0} &               &        &            &                \cr
\vdots         & \rho_{0,1} & \ddots        &        &            &                \cr
\rho_{0,m_0}   &            & \ddots        & \ddots &            &                \cr
               & \ddots     &               & \ddots & \rho_{0,0} &                \cr
\mbox{\huge 0} &            &  \rho_{0,m_0} & \cdots & \rho_{0,1} & \rho_{0,0}  
\end{matrix}
\right),
\]
respectively. Note that both $\tilde{\Delta}_{n}$ and $\Delta_{n}$ are 
sparse matrices in the sense that they have only 
$O(n)$ nonzero elements. 

By (\ref{eq:a363})--(\ref{eq:atilde362}), we have
\begin{align*}
\tilde{A}_n &= \tilde{\Delta}_{n} + \sum_{\mu=1}^K \sum_{j=1}^{m_{\mu}} Q_{\mu,j,n} \tilde{D}_{\mu,j,n}, 
\qquad n\ge m_0 + 1,
\\
A_n &= \Delta_{n} + \sum_{\mu=1}^K \sum_{j=1}^{m_{\mu}} Q^*_{\mu,j,n} D_{\mu,j,n}, 
\qquad n\ge m_0 + 1.
\end{align*}
Therefore, it is enough to give linear-time algorithms to compute 
$Q_{\mu,i,n} Y$ and $Q^*_{\mu,i,n} Y$ for $Y \in \C^{dn\times d}$ in $O(n)$ operations. 
The following two propositions provide such linear-time algorithms.

\begin{proposition}\label{prop:calcQ111}
Let $n\in\N$, $\mu\in\{1,\dots,K\}$ and $Y$ be as in (\ref{eq:Y-999}). 
We put 
$Z_{\mu,i}=Q_{\mu,i,n} Y$ for $i\in\{1,\dots,m_{\mu}\}$. 
Then the component blocks $z_{\mu,i}(s)$ of 
$Z_{\mu,i}=(z_{\mu,i}^{\top}(1),\dots,z_{\mu,i}^{\top}(n))^{\top}$ 
satisfy the following equalities:
\begin{align}
&z_{\mu,i}(n) = q_{\mu,i}(0) y_n, \qquad i \in \{1,\dots,m_{\mu}\},
\label{eq:z-recurs111}\\
&z_{\mu,1}(s) = p_{\mu} z_{\mu,1}(s+1) +q_{\mu,1}(0) y_s, \qquad s\in\{1,\dots,n-1\}
\label{eq:z-recurs222}\\
&\begin{aligned}
z_{\mu,i}(s) &= p_{\mu} z_{\mu,i}(s+1) + z_{\mu,i-1}(s) + \{q_{\mu,i}(0) -q_{\mu,i-1}(0)\} y_s,\\
&\qquad\qquad\qquad\qquad\qquad i\in\{2,\dots,m_{\mu}\}, \ s\in\{1,\dots,n-1\}.
\end{aligned}
\label{eq:z-recurs333}
\end{align}
\end{proposition}

\begin{proof}
From the definition of $Q_{\mu,i,n}$, (\ref{eq:z-recurs111}) is trivial. 
For $q_{\mu,i}(k)$ in (\ref{eq:a-224}), Pascal's rule yields 
the following recursions: 
\begin{align}
&q_{\mu,1}(k+1) = p_{\mu}q_{\mu,1}(k), \quad k \in \N\cup\{0\},
\label{eq:a-recurs111}\\
&q_{\mu,i}(k+1) = p_{\mu}q_{\mu,i}(k) +q_{\mu,i-1}(k+1), 
\quad i \in \{2,\dots,j\}, \ k \in \N\cup\{0\}.
\label{eq:a-recurs222}
\end{align}
For $s\in\{1,\dots,n-1\}$, we see, from (\ref{eq:a-recurs111}),
\[
\begin{aligned}
z_{\mu,1}(s) 
&=q_{\mu,1}(0) y_s + \sum_{t=0}^{n-s-1}q_{\mu,1}(t+1) y_{s+t+1}\\
&=q_{\mu,1}(0) y_s + p_{\mu} \sum_{t=0}^{n-s-1}q_{\mu,1}(t) y_{s+t+1}
=q_{\mu,1}(0) y_s + p_{\mu} z_{\mu,1}(s+1),
\end{aligned}
\]
and, from (\ref{eq:a-recurs222}),
\[
\begin{aligned}
z_{\mu,i}(s) 
&=q_{\mu,i}(0) y_s + \sum_{t=0}^{n-s-1}q_{\mu,i}(t+1) y_{s+t+1}\\
&= \{q_{\mu,i}(0) -q_{\mu,i-1}(0)\} y_s 
+ p_{\mu} \sum_{t=0}^{n-s-1}q_{\mu,1}(t) y_{s+t+1} 
+ \sum_{t=0}^{n-s}q_{\mu,i-1}(t) y_{s+t}\\
&=\{q_{\mu,i}(0) -q_{\mu,i-1}(0)\} y_s 
+ p_{\mu} z_{\mu,i}(s+1) + z_{\mu,i-1}(s)
\end{aligned}
\]
for $i\in\{2,\dots,j\}$. 
Thus, (\ref{eq:z-recurs222}) and (\ref{eq:z-recurs333}) follow.
\end{proof}

By Proposition \ref{prop:calcQ111}, we can compute $z_{\mu,i}(s)$ 
in the following order in $O(n)$ operations:
\[
\begin{aligned}
&z_{\mu,1}(n) \ \to \ \cdots \to \ z_{\mu,1}(1) \ \to \ z_{\mu,2}(n) \ \to \cdots \to \ z_{\mu,2}(1)\\ 
&\qquad\qquad\qquad\qquad\qquad\qquad \to \cdots \to \ z_{\mu,m_{\mu}}(n) \ \to \cdots \to z_{\mu,m_{\mu}}(1).
\end{aligned}
\]

\begin{proposition}\label{prop:calcQ222}
Let $n\in\N$, $\mu\in\{1,\dots,K\}$ and $Y$ be as in (\ref{eq:Y-999}). 
We put 
$W_{\mu,i}=Q^*_{\mu,i,n} Y$ for $i\in\{1,\dots,m_{\mu}\}$. 
Then the component blocks $w_{\mu,i}(s)$ of 
$W_{\mu,i}=(w_{\mu,i}^{\top}(1),\dots,w_{\mu,i}^{\top}(n))^{\top}$ 
satisfy the following equalities: 
\begin{align*}
&w_{\mu,i}(1) = q_{\mu,i}^*(0) y_1, \qquad i \in \{1,\dots,m_{\mu}\},
\\
&w_{\mu,1}(s+1) = \barp_{\mu} w_{\mu,1}(s) +q_{\mu,1}^*(0) y_{s+1}, \qquad s\in\{1,\dots,n-1\}
\\
&\begin{aligned}
w_{\mu,i}(s+1) &= \barp_{\mu} w_{\mu,i}(s) + w_{\mu,i-1}(s+1) + \{q_{\mu,i}^*(0) -q_{\mu,i-1}^*(0)\} y_{s+1},\\
&\qquad\qquad\qquad\qquad\qquad i\in\{2,\dots,m_{\mu}\}, \ s\in\{1,\dots,n-1\}.
\end{aligned}
\end{align*}
\end{proposition}

The proof of Proposition \ref{prop:calcQ222} is similar to that of 
Proposition \ref{prop:calcQ111}; we omit it.

By Proposition \ref{prop:calcQ222}, we can compute $w_{\mu,i}(s)$ 
in the following order in $O(n)$ operations:
\[
\begin{aligned}
&w_{\mu,1}(1) \ \to \ \cdots \to \ w_{\mu,1}(n) \ \to \ w_{\mu,2}(1) \ \to \cdots \to \ w_{\mu,2}(n)\\ 
&\qquad\qquad\qquad\qquad\qquad\qquad \to \cdots \to \ w_{\mu,m_{\mu}}(1) \ \to \cdots \to w_{\mu,m_{\mu}}(n).
\end{aligned}
\]




\appendix

\section{Proofs of Lemmas \ref{lem:bk123} and \ref{lem:w516}}\label{sec:Append-A}

As in Section \ref{sec:5}, we assume (\ref{eq:farima529}) and (\ref{eq:C}). 
We use the same notation as in Section \ref{sec:5}. 
For $K$ in (\ref{eq:hinverse162}) with (\ref{eq:hinverse163}), we assume $K\ge 1$.

We prove Lemma \ref{lem:bk123}.

\begin{proof}
(i)\ 
We assume $n\ge u\ge m_0+1$, and prove (\ref{eq:bodd123}) and (\ref{eq:beven123}) by induction. 
First, from Lemma \ref{lem:beta162}, 
\[
b_{n,u,\ell}^1 = \beta_{u+l} = \p_{u-n-1}^* (\Pi_n \Theta)^* \overline{\p}_{\ell}.
\]
Next, for $k=1,2,\dots$, we assume (\ref{eq:bodd123}). Then, by Lemma \ref{lem:beta162}, 
\[
\begin{aligned}
b_{n,u,\ell}^{2k}
&= \sum_{m=0}^{\infty} b_{n,u,m}^{2k-1} \beta_{n+1+m+\ell}^* \\
&= \p_{u-n-1}^* (\tilde{G}_n G_n)^{k-1} (\Pi_n \Theta)^* 
\left(\sum_{m=0}^{\infty} \overline{\p}_m \p_m^{\top}\right) \Pi_n \Theta \p_{\ell}\\
&=  \p_{u-n-1}^* (\tilde{G}_n G_n)^{k-1} (\Pi_n \Theta)^*  \Lambda^{\top} \Pi_n \Theta \p_{\ell}
=\p_{u-n-1}^* (\tilde{G}_n G_n)^{k-1} \tilde{G}_n \Pi_n \Theta \p_{\ell}
\end{aligned}
\]
or (\ref{eq:beven123}). 
From this as well as Lemma \ref{lem:beta162}, 
\[
\begin{aligned}
b_{n,u,\ell}^{2k+1}
&= \sum_{m=0}^{\infty} b_{n,u,m}^{2k} \beta_{n+1+m+\ell}\\
&= \p_{u-n-1}^* (\tilde{G}_n G_n)^{k-1} \tilde{G}_n \Pi_n \Theta 
\left(\sum_{m=0}^{\infty} \p_{m} \p_{m}^*\right) (\Pi_n \Theta)^* \overline{\p}_{\ell}\\
&= \p_{u-n-1}^* (\tilde{G}_n G_n)^{k-1} \tilde{G}_n \Pi_n \Theta \Lambda (\Pi_n \Theta)^* 
\overline{\p}_{\ell}
= \p_{u-n-1}^* (\tilde{G}_n G_n)^{k} (\Pi_n \Theta)^* \overline{\p}_{\ell}
\end{aligned}
\]
or (\ref{eq:bodd123}) with $k$ replaced by $k+1$. 
Thus (\ref{eq:bodd123}) and (\ref{eq:beven123}) follow.

(ii)\ 
We assume $1\le u\le n-m_0$, and prove (\ref{eq:til-bodd123}) and (\ref{eq:til-beven123}) by induction. 
First, from Lemma \ref{lem:beta162}, 
\[
\tilde{b}_{n,u,\ell}^1 = \beta_{n+1-u+\ell}^* = \p_{-u}^{\top} \Pi_n \Theta \p_{\ell}.
\]
Next, for $k=1,2,\dots$, we assume (\ref{eq:til-bodd123}). Then, by Lemma \ref{lem:beta162}, 
\[
\begin{aligned}
\tilde{b}_{n,u,\ell}^{2k}
&= \sum_{m=0}^{\infty} \tilde{b}_{n,u,m}^{2k-1} \beta_{n+1+m+\ell}\\
&= \p_{-u}^{\top} (G_n \tilde{G}_n)^{k-1} \Pi_n \Theta 
\left(\sum_{m=0}^{\infty} \p_m \p_m^* \right) (\Pi_n \Theta)^* \overline{\p}_{\ell}\\
&=  \p_{-u}^{\top} (G_n \tilde{G}_n)^{k-1} \Pi_n \Theta  \Lambda (\Pi_n \Theta)^* \overline{\p}_{\ell}
=\p_{-u}^{\top} (G_n \tilde{G}_n)^{k-1} G_n (\Pi_n \Theta)^* \overline{\p}_{\ell}
\end{aligned}
\]
or (\ref{eq:til-beven123}). 
From this as well as Lemma \ref{lem:beta162}, 
\[
\begin{aligned}
\tilde{b}_{n,u,\ell}^{2k+1}
&= \sum_{m=0}^{\infty} \tilde{b}_{n,u,m}^{2k} \beta_{n+1+m+\ell}^*
= \p_{-u}^{\top} (G_n \tilde{G}_n)^{k-1} G_n (\Pi_n \Theta)^* 
\left(\sum_{m=0}^{\infty} \overline{\p}_{m} \p_{m}^{\top}\right) \Pi_n \Theta \p_{\ell}\\
&= \p_{-u}^{\top} (G_n \tilde{G}_n)^{k-1} G_n (\Pi_n \Theta)^* \Lambda^{\top} \Pi_n \Theta 
\p_{\ell}
= \p_{-u}^{\top} (G_n \tilde{G}_n)^{k} \Pi_n \Theta \p_{\ell}
\end{aligned}
\]
or (\ref{eq:til-bodd123}) with $k$ replaced by $k+1$. 
Thus (\ref{eq:til-bodd123}) and (\ref{eq:til-beven123}) follow.
\end{proof}

To prove Lemma \ref{lem:w516}, we need some propositions.

\begin{proposition}\label{prop:dif123}
For $m, n\in \Z$, $i, j \in\N\cup\{0\}$ and $x, y\in \D$, we have
\[
\begin{aligned}
&\frac{1}{i! j!} \left(\frac{\partial }{\partial x}\right)^{i}
\left(\frac{\partial }{\partial y}\right)^{j} 
\frac{x^m y^n}{1-xy}\\
&\qquad\qquad\qquad =\sum_{q=0}^{i} \sum_{r=0}^{j}  
\binom{n}{j-r} \binom{q+r}{q} \binom{m+r}{i-q} 
  \frac{x^{m+q+r-i} y^{n+q+r-j}}{(1-xy)^{q+r+1}}.
\end{aligned}
\]
\end{proposition}

\begin{proof}
Let $m,n\in \Z$, $i, j \in\N\cup\{0\}$ and $x, y\in \D$. 
Then, we have
\[
\begin{aligned}
\frac{1}{j!} 
\left(\frac{\partial }{\partial y}\right)^{j}
\frac{y^n}{1-xy}
&=
\sum_{r=0}^{j}  
\left\{
\frac{1}{r!}  \left(\frac{\partial }{\partial y}\right)^{r}\frac{1}{1-xy}\right\}
\left\{ \frac{1}{(j-r)!} \left(\frac{\partial }{\partial y}\right)^{j-r}y^{n} \right\}
\\
&=
\sum_{r=0}^{j} 
\binom{n}{j-r} \frac{x^{r}y^{n+r-j}}{(1-xy)^{r+1}},
\end{aligned}
\]
hence
\[
\begin{aligned}
&\frac{1}{i! j!} \left(\frac{\partial }{\partial x}\right)^{i}\left(\frac{\partial }{\partial y}\right)^{j} 
\frac{x^my^n}{1-xy}
\\
&=
\sum_{r=0}^{j} \binom{n}{j-r} y^{n+r-j} \sum_{q=0}^{i}  
\left\{
\frac{1}{q!}  \left(\frac{\partial }{\partial x}\right)^{q}\frac{1}{(1-xy)^{r+1}}\right\}
\left\{ \frac{1}{(i-q)!} \left(\frac{\partial }{\partial x}\right)^{i-q} x^{m+r} \right\}
\\
&=
\sum_{r=0}^{j} \binom{n}{j-r} y^{n+r-j} \sum_{q=0}^{i}  
\left\{
\binom{q+r}{q}  \frac{y^{q}}{(1-xy)^{q+r+1}}\right\}
\left\{ \binom{m+r}{i-q} x^{m+q+r-i} \right\}\\
&= \sum_{q=0}^{i} \sum_{r=0}^{j}  
\binom{n}{j-r} \binom{q+r}{q} \binom{m+r}{i-q} 
  \frac{x^{m+q+r-i} y^{n+q+r-j}}{(1-xy)^{q+r+1}}.
\end{aligned}
\]
Thus, the proposition follows.
\end{proof}

\begin{proposition}\label{prop:sum123}
For $n\in \Z$, $i, j \in\N\cup\{0\}$ and $x, y\in \D$, we have
\[
\sum_{\ell=0}^{\infty} \binom{n+\ell}{i} \binom{j + \ell}{j} x^{n+\ell-i} y^{\ell}
= \sum_{q=0}^{i} \sum_{r=0}^{j}  
\binom{j}{r} \binom{r+q}{q} \binom{n+r}{i-q} 
  \frac{x^{n+r+q-i} y^{r+q}}{(1-xy)^{r+q+1}}.
\]
\end{proposition}

\begin{proof}
Let $n\in \Z$, $i, j \in\N\cup\{0\}$ and $x, y\in \D$. 
Since 
$x^ny^j/(1-xy)=\sum_{\ell=0}^{\infty} x^{n+\ell}y^{j+\ell}$, 
we have
\[
\frac{1}{i! j!} \left(\frac{\partial }{\partial x}\right)^{i}\left(\frac{\partial }{\partial y}\right)^{j} 
\frac{x^ny^j}{1-xy}
=\sum_{\ell=0}^{\infty} \binom{n+\ell}{i} \binom{j+\ell}{j} x^{n+\ell-i} y^{\ell}.
\]
On the other hand, by Proposition \ref{prop:dif123}, 
we have
\[
\frac{1}{i! j!} \left(\frac{\partial }{\partial x}\right)^{i}\left(\frac{\partial }{\partial y}\right)^{j} 
\frac{x^ny^j}{1-xy}= \sum_{q=0}^{i} \sum_{r=0}^{j}  
\binom{j}{r} \binom{r+q}{q} \binom{n+r}{i-q} 
  \frac{x^{n+r+q-i} y^{r+q}}{(1-xy)^{r+q+1}}.
\]
Comparing, we obtain the proposition.
\end{proof}

We are ready to prove Lemma \ref{lem:w516}.

\begin{proof}
By (\ref{eq:a363})--(\ref{eq:atilde362}) and Proposition \ref{prop:sum123}, we have, 
for $n\in\Z$, $\mu \in \{1,\dots,K\}$ and $i \in \{1,\dots,m_{\mu}\}$, 
\[
\begin{aligned}
\sum_{\ell=0}^{\infty} p_{\mu,i}(\ell - n) a_{\ell}
&= \sum_{\ell=0}^{m_0} p_{\mu,i}(\ell - n) \rho_{0,\ell}\\
&\qquad 
+\sum_{{\nu}=1}^{K} \sum_{j=1}^{m_{\nu}} 
\left\{
\sum_{\ell=0}^{\infty}
\binom{\ell - n}{i-1} \binom{\ell+j-1}{j-1}  p_{\mu}^{\ell-i+1 - n} \barp_{\nu}^{\ell} 
\right\}
\rho_{\nu,j}\\
&= \sum_{\ell=0}^{m_0} p_{\mu,i}(\ell - n) \rho_{0,\ell}
+ \sum_{{\nu}=1}^{K} \sum_{j=1}^{m_{\nu}} \varphi_n^{{\mu},{\nu}}(i, j) \rho_{\nu,j}
\end{aligned}
\]
and
\[
\begin{aligned}
\sum_{\ell=0}^{\infty} \barp_{\mu,i}(\ell - n) \tilde{a}_{\ell}
&= \sum_{\ell=0}^{m_0} \barp_{\mu,i}(\ell - n) 
\tilde{\rho}_{0,\ell}\\
&\qquad 
+\sum_{{\nu}=1}^{K} \sum_{j=1}^{m_{\nu}} 
\left\{
\sum_{\ell=0}^{\infty}
\binom{\ell - n}{i-1} \binom{\ell+j-1}{j-1}  \barp_{\mu}^{\ell-i+1 - n} p_{\nu}^{\ell} 
\right\}
\tilde{\rho}_{\nu,j}\\
&= \sum_{\ell=0}^{m_0} \barp_{\mu,i}(\ell - n) 
\tilde{\rho}_{0,\ell}
+ \sum_{{\nu}=1}^{K} \sum_{j=1}^{m_{\nu}} \overline{\varphi}_n^{{\mu},{\nu}}(i, j) \tilde{\rho}_{\nu,j}.
\end{aligned}
\]
Thus, the lemma follows.
\end{proof}


\section*{Acknowledgments}
The author is grateful to two referees for their valuable comments and helpful suggestions.

%
%



\end{document}